\documentclass{amsart}
\usepackage{amsfonts}
\usepackage{amssymb}
\usepackage{amsmath}
\usepackage{amsrefs}
\usepackage{mathrsfs}

\usepackage{color}
\usepackage[dvipsnames]{xcolor}
\usepackage{tikz}
\usepackage{setspace}
\usepackage{multicol}
\usetikzlibrary{calc,intersections,decorations.pathreplacing}
\usepackage{ragged2e}
\usepackage[linktocpage=true]{hyperref}
\usepackage{pgfplots}

\usetikzlibrary{arrows.meta}
\usetikzlibrary{arrows}

\newtheorem{thm}{Theorem}[section]
\newtheorem{cor}[thm]{Corollary}
\newtheorem{lem}[thm]{Lemma}
\newtheorem{prop}[thm]{Proposition}

\newtheorem*{theo}{Theorem}

\theoremstyle{definition}

\newtheorem{defn}[thm]{Definition}
\newtheorem{prop-def}[thm]{Proposition--Definition}
\newtheorem{claim}[thm]{Claim}

\newtheorem{sublem}[thm]{Fact}
\newtheorem{exmp}[thm]{Example}

\theoremstyle{remark}
\newtheorem{rmk}[thm]{Remark}

\newcommand\dashto{\mathrel{
  -\mkern-6mu{\to}\mkern-20mu{\color{white}\bullet}\mkern12mu
}}

\newcommand{\nn}{{\mathbb N}}
\newcommand{\zz}{{\mathbb Z}}

\newcommand{\cc}{{\mathbb C}}
\newcommand{\kk}{{\mathbb C}}

\newcommand{\modu}{\mathscr}
\newcommand{\sh}{\mathcal}

\usetikzlibrary{fit}

\usetikzlibrary{fit}

\hypersetup{ colorlinks=true, citecolor=cyan, linkcolor=blue, urlcolor=blue}

\begin{document}

\title[The irreducibility of the generalized Severi varieties]{The irreducibility of the generalized Severi varieties}


\author[A. Zahariuc]{Adrian Zahariuc}
\address{Department of Mathematics, UC Davis, One Shields Ave, Davis, CA 95616}
\email{azahariuc@math.ucdavis.edu}





\begin{abstract}
We give an inductive proof that the generalized Severi varieties -- the varieties which parametrize (irreducible) plane curves of given degree and genus, with a fixed tangency profile to a given line at several general fixed points and several mobile points -- are irreducible.
\end{abstract}

\maketitle

\tableofcontents

\section*{Introduction}
The Severi variety $\smash{ V_{d,g} \subset {\mathbb P}^N = |{\sh O}_{{\mathbb P}^2}(d)| }$ parametrizes degree $d$ genus $g$ irreducible plane curves with at worst nodal singularities. Let $\smash{ \overline{V}_{d,g} }$ be the closure inside $\smash{ |{\sh O}_{{\mathbb P}^2}(d)| }$. These parameter spaces were introduced almost a century ago by Francesco Severi with the purpose of analyzing the moduli space $M_g$ of abstract curves via the rational map $\smash{ \overline{V}_{d,g} \dashto M_g }$. In \emph{Anhang F} to \emph{Vorlesungen \"{u}ber algebraische Geometrie}, he claimed the following properties of $V_{d,g}$.
\begin{theo}[\cite{[Se21], [Ha86]}] 
$V_{d,g}$ is a smooth, nonempty, irreducible variety of dimension $3d+g-1$.
\end{theo}
Several years later, Severi's proof of irreducibility was acknowledged to be incorrect and only in the mid 80s, a proof using an intricate degeneration argument was given by Joe Harris ~\cite{[Ha86]}. The three main steps in the proof of irreducibility are the following: 

(1) locally near rational nodal curves (in the complex analytic topology), $\smash{ \overline{V}_{d,g} }$ consists of several smooth branches, each branch corresponding to smoothing a specific set of $g$ nodes of the rational curve; 

(2) monodromy acts fully on the set of nodes of the rational curve;  and 

(3) any irreducible component of $\smash{ \overline{V}_{d,g} }$ contains a nodal curve of strictly smaller genus, so, inductively,  it contains a nodal rational curve, which implies irreducibility by the previous steps. 

The third step is by far the most involved one and, although the argument by degeneration used to prove it may be less transparent, it is precisely this part of the proof which led to substantial progress in various directions over the following years. 

One such direction was the solution by Ran ~\cite{[Ran89]} and Caporaso and Harris ~\cite{[CH98]} of the enumerative problem of counting degree $d$ genus $g$ plane curves passing through $3d+g-1$ general points in the plane. Obviously, the problem amounts to computing the degrees of the Severi varieties. The solution in ~\cite{[CH98]} involves a refinement of the degeneration argument used to prove step (3) above and the answer comes in the form of a recursive formula involving the degrees of the so-called generalized Severi varieties: the varieties parametrizing (possibly reducible!) plane curves of given degree and geometric genus, with a fixed tangency profile to a given line at several general fixed points and several mobile points.

Although in the case of the projective plane the proof of irreducibility predates the computation of degrees, in terms of further developments and generalizations to other surfaces, the situation is the complete opposite. On one hand, there have been only a few further developments concerning the irreducibility problem ~\cite{[Ty06], [Te10]}. On the other hand, 
there has been a great amount of work and progress on the enumerative side, partly due to the availability of powerful tools such as (but not limited to) J. Li's degeneration formula ~\cite{[Li02]}, which allows one to systematically compute Caporaso--Harris type formulas.

Given that the genesis of the generalized Severi varieties is quite intimately linked with the irreducibility problem, it is natural to ask whether they are irreducible. The answer is obviously no and the reason is that the curves are allowed to be reducible. However, this requirement (or lack thereof) is mostly a cosmetic one meant to simplify the combinatorics in ~\cite{[CH98]}. It turns out that if we instead insist in the definition that the curves are irreducible (Definition \ref{main def}), then the generalized Severi varieties are indeed irreducible, as expected. To the best of the author's knowledge, this statement doesn't appear anywhere in the literature, although it seems to lie within the realm of what's provable by adapting the arguments in the famous proof ~\cite{[Ha86]} by J. Harris discussed above. The objective purpose of the paper is to prove this statement (Theorem \ref{main theorem}). However, we will not prove this statement by adapting the existing methods, so the subjective purpose is to outline and test a different approach to the irreducibility problem. 

Although in this paper we will only be concerned with the projective plane ${\mathbb P}^2$, we briefly digress to illustrate (in the case of K3s) the fundamental difficulties with generalizing the irreducibility statement to other surfaces. It is an open problem whether the Severi variety of genus $h \leq g$ curves in the primitive class of a general degree $2g-2$ K3 surface is irreducible. The analogue of step (3) above is known to hold by work of X. Chen ~\cite{[Ch01], [Ch16]}, that is, every irreducible component of the (closure of the) Severi variety contains nodal curves of strictly smaller genus. At first glance, this makes an analogous approach by induction on $h$ look attractive. The main problem is that the base case fails: the primitive class contains finitely many rational curves, so the genus $0$ Severi variety is certainly not irreducible. If we instead try to start at $h=1$, we immediately run into the problem that we lack the methods to prove this case. 

The bottom line is that, in order to approach the irreducibility problem in general, we need a method which avoids the reduction to the genus zero case.

The main novelty in the current paper is the observation that a part of the theory underlying the degeneration formula ~\cite{[Li01], [Li02]}, specifically the ingredient which allows one to separate the contributions from different topological profiles in the central fiber, offers an alternative approach to the irreducibility problem. Morally speaking, this crucial technical observation (\cite[Lemma 3.4]{[Li02]}, also Fact \ref{key cite}) goes back to the versal deformation space of a node. However, the way the idea is expressed in the language of relative stable maps is optimal for our needs.

The central idea of the proof is that, after degenerating the ambient plane\footnote{~\cite{[Ha86], [CH98]} express the analysis in terms of plane curves forced to "split off" a copy of a line. We will use a related point of view by degenerating the plane to the union of a plane and an ${\mathbb F}_1$ surface.} and thereby the moduli space, the key technical ingredient  alluded to in the previous paragraph allows us to argue that, under certain circumstances, intersecting components of the degenerate moduli space "coalesce"\footnote{The word "coalesce" is suggestive, but strictly speaking incorrect because of multiplicities issues. A priori we're only proving that some "layers" of the nonreduced components coalesce and this turns out to be sufficient.} (Proposition \ref{key trick}) when moving away from the degenerate central fiber. This reduces the proof to showing that a certain graph bookkeeping the discrete data labeling the various components and the desirable transitions between them is connected (Proposition \ref{landscape is connected}), which turns out to be completely straightforward. The reader may find the general principle to be perhaps reminiscent of that in the classic ~\cite{[DeMu69]}. A more detailed outline of the proof is given in the preamble to \S5. 

In principle, our approach works exclusively by induction on degree, bypassing the reduction to the high-degree-low-genus case, which, as explained above, represents the main obstacle to applying the usual methods in the case of irrational surfaces, where the problem appears to be universally open (even in the presumably simplest case of $E \times {\mathbb P}^1$ ~\cite{[Bu14]}). High degree rational curves do show up in the paper but in relation to some non-emptiness issues rather than irreducibility issues and their role doesn't seem to be central. In future work, we hope to investigate the extent to which the methods introduced in this paper are applicable to irrational surfaces.

We work over the field $\kk$ of complex numbers, as in ~\cite{[CH98], [Li01], [Se21]}.

\bigskip

\noindent \textit{Acknowledgments.} I would like to thank Brian Osserman for both encouragement and substantial mathematical help while working on this project and Joe Harris for introducing me to these problems.
 
\section{Notation and preliminaries} 
We will consider three flavors of compactifations: (1) the classical compactification obtained by taking Zariski closure in $|{\sh O}_{{\mathbb P}^2}(d)|$; (2) a compactification inside the Kontsevich moduli space of stable maps (reminiscent of the one used by Gathmann ~\cite{[Ga01]}, but less sophisticated); and (3) the relative stable maps compactification of Li ~\cite{[Li01], [Li02]}. All three will be used in the proof.

\subsection{Generalized Severi varieties}\label{main def} Our first concern is to introduce the main objects studied in this paper, the generalized Severi varieties, as outlined in the introduction. In short, the only essential difference compared to ~\cite{[CH98]} is that we require the curves to be irreducible. There are also cosmetic differences, such as our choice of initially requiring the curves to have at worst nodal singularities, but, after taking Zariski closures, all these differences turn out to be irrelevant \cite[Proposition 2.2]{[CH98]}. Let $\smash{ {\mathbb P}^N = |{\sh O}_{{\mathbb P}^2}(d)| }$ be the projective space parametrizing all degree $d$ plane curves. All generalized Severi varieties will be taken relative to the line at infinity $L = \{Z = 0\}$ in ${\mathbb P}^2[X:Y:Z]$. 

To define the generalized Severi varieties, we need to fix the following data: the genus $\smash{ g \leq {d-1 \choose 2} }$, $n, k$ nonnegative integers, $n$ distinct fixed points $p_1,p_2,...,p_n \in L$, and multiplicities $\smash{m^v_1,m^v_2,...,m^v_k}$ and $\smash{m^f_1,m^f_2,...,m^f_n}$ collectively adding up to $d$. For simplicity of notation, we encode the $n$ fixed contact points and the orders of contact at the fixed and variable points as tuples in the following way: $\Lambda = (p_1,p_2,...,p_n)$, $\smash{ {\mathbf m}_f = (m^f_1,m^f_2,...,m^f_n)} $ and $\smash{ {\mathbf m}_v = (m^v_1,m^v_2,...,m^v_k) }$. 

\begin{defn}\label{nice curve}
We say that an \emph{irreducible} degree $d$ plane curve $C \subset {\mathbb P}^2$ is nice relative to the given data if $C$ has the following properties:

(1) $C$ has precisely ${d-1 \choose 2} - g$ nodal singularities occurring away from the line $L$ and no other singularities;

(2) there are $k$ distinct points $\xi_1,\xi_2,...,\xi_k \in C^{\mathrm{ns}}({\kk})$ which are different from the $p_i$, such that
$$ (Z)|_C = \sum_{i=1}^{n} m^f_i p_i + \sum_{j=1}^{k} m^v_j \xi_j. $$
Of course, by $(Z)|_C$ we mean the scheme-theoretic vanishing of the restriction to $C$ of the section $Z \in \mathrm{H^0}({\sh O}_{{\mathbb P}^2}(1))$ defining $L$. 
\end{defn}

\begin{prop-def} 
The \emph{generalized Severi variety} is the locally closed subvariety $\smash{ V_{d,g}(\Lambda,{\mathbf m}_f;{\mathbf m}_v) \subset {\mathbb P}^N }$ whose set of $\kk$-points is the set of nice curves relative to the given data. 
\end{prop-def}

By a slight abuse of language, we will usually call its closure $\smash{ \overline{V}_{d,g}(\Lambda,{\mathbf m}_f;{\mathbf m}_v ) }$ in ${\mathbb P}^N$ a generalized Severi variety as well. The distinction will always be clear from context. Finally, a comment on the assumption that the fixed contact points are distinct. Contrary to the definition above, we will consider cases where two points $p_i$ and $p_j$ coincide and this is actually absolutely crucial in the proof of the main theorem. However, in such cases, the curve will be regarded as an element of a generalized Severi variety where $p_i$ and $p_j$ are mobile contact points, as our subsequent definitions will make clear.

Throughout most of the paper, we will be concerned mostly with the case when all contact points are fixed, that is, the case $k=0$. In this case, we will write simply $\smash{ \overline{V}_{d,g}(\Lambda,{\mathbf m}) }$ for the generalized Severi variety, where $\smash{ {\mathbf m} = {\mathbf m}_f }$ and ${\mathbf m}_v$ is empty, so it will be omitted from our notation. Similarly, if all contact points are mobile, we will write $\smash{ \overline{V}_{d,g}({\mathbf m}) }$, where ${\mathbf m} = {\mathbf m}_v$ and ${\mathbf m}_f$ and $\Lambda$ are understood to be empty.

\begin{exmp}\label{hs algebra}
For each $d$, there exists a nice degree $d$ rational curve with a single contact point of multiplicity $d$ with the line at infinity. For instance, if $d \geq 3$, we may take the rational curve given parametrically by the map $\smash{ {\mathbb P}^1 \stackrel{\alpha}{\to} {\mathbb P}^2[X:Y:Z] }$
$$ [u:w] \mapsto [u^d+u^{d-1}w+uw^{d-1}+w^d:u^d:w^d] $$
Indeed, it is easy to verify that the points $[\sqrt[d-2]{-1-\epsilon-...-\epsilon^{d-2}}:1] \in {\mathbb P}^1$ are preimages of nodes of the plane curve, where $1 \neq \epsilon \in \kk$ is a $d$-th root of unity and that $\mathrm{d}\alpha$ is injective at all points. Clearly, $[1:1:0]$ is the only point where the plane curve intersects the line at infinity.
\end{exmp}

\begin{thm}[\cite{[CH98]} and folklore]\label{general thm} Let $\Lambda$ be  a general collection of points and all discrete data as above. Then the generalized Severi variety $\smash{ V_{d,g}(\Lambda,{\mathbf m}_f;{\mathbf m}_v ) }$ is nonempty, smooth and equidimensional of dimension $2d+g+k-1$. The degree (of the closure) is computed by the algorithm described in \cite[\S\S1.4]{[CH98]}.
\end{thm}

The dimension computation follows from the analogous statement in \cite{[CH98]}. Smoothness is a fairly straightforward deformation theoretic computation treating the curve as an unramified map $\smash{\nu: \tilde{C} \to {\mathbb P}^2}$. The space of first order deformations of $C$ is the space $\smash{ \mathrm{H}^0(\tilde{C},{\sh N}_\nu(-D)) }$, where
\begin{equation}\label{stuff2.2}
D = \sum_{i=1}^{n} m^f_i p_i + \sum_{j=1}^{k} (m^v_j-1) \xi_j \in \mathrm{Div}(\tilde{C}),
\end{equation}
cf. ~\cite{[CH98], [HM01]}. The fact that $\smash{ \mathrm{H}^1(\tilde{C},{\sh N}_\nu(-D)) = 0}$ turns out to be trivial for degree reasons, which implies smoothness. Indeed, we have $\deg {\sh N}_{\nu}(-D) = 2g-2+2d+k$ since ${\sh N}_{\nu} = \omega_{\tilde{C}} \otimes \nu^*{\sh O}_{{\mathbb P}^2}(3)$
from the short exact sequence for the normal sheaf of $\nu$, hence $h^0({\sh N}_{\nu}(-D)) = 2d+g-1+k$. It goes without saying that the closure $\smash{ \overline{V}_{d,g}(\Lambda,{\mathbf m}_f;{\mathbf m}_v ) }$ is very badly singular. For a discussion of the singularities of the usual Severi varieties in codimension one, please see ~\cite{[DiHa88]}.

Let $\smash{ q:\overline{V}_{d,g}(\Lambda,{\mathbf m}_f;{\mathbf m}_v) \dashto L^{[k]} }$ be the map giving the "residual" set-theoretic intersection with $L$. This is a rational map, defined on a dense open subset of the source. Similarly, the $q$-relative tangent space to $[C]$ is $\smash{ \mathrm{H}^0(\tilde{C},{\sh N}_\nu(-D_f)) }$, where
\begin{equation}
D_f = \sum_{i=1}^{n} m^f_i p_i + \sum_{j=1}^{k} m^v_j \xi_j \in \mathrm{Div}(\tilde{C}).
\end{equation}
A similar calculation shows that the dimension of this space is equal to the $2d+g-1$, meaning that $q$ is smooth. A trivial consequence is the following remark.

\begin{rmk}\label{early def th}
The restriction of the map $\smash{q }$ to every hypothetical irreducible component of $\smash{\overline{V}_{d,g}(\Lambda,{\mathbf m}_f;{\mathbf m}_v) }$ is dominant. 
\end{rmk}
 
The only part of Theorem \ref{general thm} which doesn't seem to be stated explicitly or implicitly anywhere in the literature is the non-emptiness part. Of course, since the degrees of the generalized Severi varieties are known, one can perhaps argue that non-emptiness is no longer an interesting question.\footnote{The point is that the degree is computable, but the formula doesn't seem to be explicit enough to allow us to check immediately that the number is not equal to zero, partly because the generalized Severi varieties in the sense of Caporaso and Harris have additional components.} Alternatively, it is possible to use methods essentially identical to those used in the first half of the proof of Proposition \ref{UC dopelganger} to rigorously establish non-emptiness. Briefly, we start with a reducible curve whose components are analogous to the curve in Example \ref{hs algebra} above and deform it smoothing some of the nodes and preserving others, while at the same time keeping the contact with $L$ fixed. We leave this exercise to the reader.

The main purpose of the paper is to prove the following theorem.

\begin{thm}[Main Theorem]\label{main theorem}
Let $\Lambda$ be a general collection of points and all discrete data as above. Then the generalized Severi variety $\smash{ V_{d,g}(\Lambda,{\mathbf m}_f;{\mathbf m}_v ) }$ is irreducible.
\end{thm}

Note that Remark \ref{early def th} implies that, in order to prove the Main Theorem \ref{main theorem}, it suffices to prove the special case when all contact points are fixed, that is, it suffices to prove that $\smash{\overline{V}_{d,g}(\Lambda,{\mathbf m}) }$ is irreducible for all ${\mathbf m} = {\mathbf m}_f$ and general $\Lambda$. Indeed, assume that the latter statement holds. Then the map $\smash{q}$ has generically irreducible fibers of the same dimension for all ${\mathbf m}_f$, ${\mathbf m}_v$ and general $\Lambda$, so $\smash{ \overline{V}_{d,g}(\Lambda,{\mathbf m}_f;{\mathbf m}_v) }$ has a unique irreducible component which $q$-dominates $\smash{ L^{[k]} }$. By Lemma \ref{early def th}, it can't have any other irreducible components, hence it is irreducible. 

Let $\smash{ \hat{V}_{d,g} ({\mathbf m}) \to V_{d,g} ({\mathbf m}) }$ be the degree $k!$ \'{e}tale cover of $\smash{ V_{d,g} ({\mathbf m}) }$ which parametrizes nice curves as before, but whose variable intersection points with $L$ are now indexed. Of course, $\smash{ \hat{V}_{d,g} ({\mathbf m}) }$ may be disconnected, but all connected components are isomorphic. By the same reasoning as above, the irreducibility of these connected components also follows from the a priori particular case when all contact points are fixed. 

\subsection{The stable maps compactification} In this subsection, we outline a straightforward alternate compactification of the generalized Severi varieties, this time using the language of stable maps. However, we will only treat the case when all contact points with $L$ are fixed. Consider the same given data as in the previous subsection, with the assumption $k=0$. We assume again that $p_i \neq p_j$ for all $i \neq j$. Let $\smash{ \overline{\modu M}_{g,n}({\mathbb P}^2,d) }$ be the Kontsevich moduli space of degree $d$ arithmetic genus $g$ stable maps into ${\mathbb P}^2$ with $n$ marked points. We have an evaluation morphism at the $n$ marked points
$$ \mathrm{ev}: \overline{\modu M}_{g,n}({\mathbb P}^2,d) \longrightarrow ({\mathbb P}^2)^n. $$
Let $\smash{\overline{\modu M}_{g}({\mathbb P}^2,d|\Lambda,{\mathbf m})}$ be the closed substack of $\smash{ \overline{\modu M}_{g,n}({\mathbb P}^2,d) }$ where the $i$th marked point maps to $p_i$ and the pullback to the source of the map of the section $Z \in \mathrm{H}^0({\sh O}_{{\mathbb P}^2}(1))$ vanishes to order at least $m_i$ at the $i$th marked point. However, this includes many unwanted situations such as those when the marked points lie on components mapping to $L$. At the expense of unpleasant technical complications later in the paper, we will rectify this simply by taking closures as follows.

Let $C_{d,g}(\Lambda,{\mathbf m}) \subset {\mathbb P}^2 \times V_{d,g}(\Lambda,{\mathbf m}) \to V_{d,g}(\Lambda,{\mathbf m})$ be the universal family of embedded curves over the generalized Severi variety constructed in the previous subsection and $\smash{ \widetilde{C}_{d,g}(\Lambda,{\mathbf m}) }$ its normalization. It easy to see that the normalized universal family $\smash{ \widetilde{C}_{d,g}(\Lambda,{\mathbf m}) \to V_{d,g}(\Lambda,{\mathbf m}) }$ is smooth, and of course the contact points with $L$ remain different, so we obtain a map
\begin{equation}\label{sketchy immersion}
V_{d,g}(\Lambda,{\mathbf m}) \longrightarrow \overline{\modu M}_{g,n}({\mathbb P}^2,d).
\end{equation}
The key observation is that this map is an immersion. The subtle part of the verification is to check that the map is unramified. However, all this is implicit in the description of the tangent space to the generalized Severi variety (\ref{stuff2.2}). We leave the details to the reader, but we note that without the assumption that $V_{d,g}(\Lambda,{\mathbf m})$ parametrizes only curves which are at worst nodal, the statement is false. 

We define $\smash{ \overline{\modu S}_{g}({\mathbb P}^2,d|\Lambda,{\mathbf m}) }$ to be the closure of the image of (\ref{sketchy immersion}), i.e. we take the closure inside the underlying topological space and endow it with the reduced closed substack structure. Thus $\smash{ \overline{\modu S}_{g}({\mathbb P}^2,d|\Lambda,{\mathbf m}) }$ is a closed substack of $\smash{ \overline{\modu M}_{g,n}({\mathbb P}^2,d) }$ by construction. The letter "${\modu S}$" means "coming from the Severi variety."

\subsection{The relative stable maps compactification} We refer the reader to ~\cite{[Li01], [Li02]} for the general theory of relative stable maps.  Let ${\mathfrak P}^\mathrm{rel}$ be the stack of expansions of the pair $({\mathbb P}^2, L)$, as defined in ~\cite{[Li01]}. Given topological data $\Gamma$ consisting of the following information: an edgeless graph $V(\Gamma)$ with roots indexed by an indexing set $R$, with the incidence relation between vertices and roots described by a surjective function $\lambda:R \to V(\Gamma)$, two functions $d_{\mathfrak P}:V(\Gamma) \to {\mathbb N}^*$ and $g:V(\Gamma) \to {\mathbb N}$ representing the degrees and arithmetic genera of the connected components of the relative stable maps such that $\sum_{v \in V(\Gamma)} d_{\mathfrak P}(v) = d$ and $ \sum_{v \in V(\Gamma)} g(v) - |V(\Gamma)| + 1 = g$
and a map $\mu:R \to \nn^*$ giving the orders of tangency with the line $L$ at the distinguished marked points, such that $\sum_{\lambda(\alpha) = v} \mu(\alpha) = d_{\mathfrak P}(v)$ for each vertex $v$, a moduli space of relative stable maps $\smash{ \overline{\modu M}({\mathfrak P}^\mathrm{rel},\Gamma) }$ can be defined following ~\cite{[Li01]}. Note that we are using the symbol $\smash{\overline{\modu M}}$ instead of the symbol ${\mathfrak M}$ used in ~\cite{[Li01]}. For purely notational purposes which will become clear in the next sections, it is useful to allow an abstract indexing set $R$ instead of $\{1,2,...,n\}$. There is an evaluation map at the distinguished marked points 
$$ {\mathbf q}:\overline{\modu M}({\mathfrak P}^\mathrm{rel},\Gamma) \longrightarrow L^{|R|}.$$
We have an obvious map from the product of the Severi varieties with indexed only-variable contact points with the line $L$
$$ \prod_{v \in V(\Gamma)} \hat{V}_{d_{\mathfrak P}(v),g(v)}\left( \mu|_{\lambda^{-1}(v)} \right) \longrightarrow \overline{\modu M}({\mathfrak P}^\mathrm{rel},\Gamma). $$
We define $\overline{\modu S}({\mathfrak P}^\mathrm{rel},\Gamma)$ to be the closure of the image of this map. More generally, given a subset of indices $R_f \subseteq R$, let $\smash{ \Lambda = (p_{\alpha})_{\alpha \in R_f} \in L }$ and define
\begin{equation}\label{defrel}
\overline{\modu M}({\mathfrak P}^\mathrm{rel},\Gamma|\Lambda) = {\mathbf q}^{-1}\left(\{ \Lambda \} \times L^{|R|-|R_f|} \right)  
\end{equation}
and 
$\smash{ \overline{\modu S}({\mathfrak P}^\mathrm{rel},\Gamma|\Lambda) = \overline{\modu M}({\mathfrak P}^\mathrm{rel},\Gamma|\Lambda) \cap \overline{\modu S}({\mathfrak P}^\mathrm{rel},\Gamma) }$ to be the closed substacks parametrizing the relative stable maps which send the desired distinguished marked points to the fixed $p_{\alpha}$'s. We will use this construction most often for $R_f = R$. 

If $R_f = R$ and the $p_{\alpha}$ are general, then it is still true that $\smash{ \overline{\modu S}({\mathfrak P}^\mathrm{rel},\Gamma|\Lambda) }$ is a compactification of the product $\smash{ \prod_{v \in V(\Gamma)} V_{d_{\mathfrak P}(v),g(v)}\left( \Lambda|_{\lambda^{-1}(v)}, \mu|_{\lambda^{-1}(v)} \right) }$. However, thanks to the way we constructed the space, we may speak of $\overline{\modu S}({\mathfrak P}^\mathrm{rel},\Gamma|\Lambda)$ even when the entries of $\Lambda$ are not all distinct. This will be crucial in the next section, when we investigate precisely what happens when two of the points $p_\alpha$ collide.

\section{Doppelg\"{a}ngers} 

\subsection{General definitions} A "doppelg\"{a}nger" is a lookalike or a ghostly double of a person. A doppelg\"{a}nger of a stable map into $(Y,D)$ is roughly another stable map, whose image is the same $1$-cycle in $Y$ yet has a different configuration of contracted ("ghost") components mapping into $D$, which may or may not change its arithmetic genus. However, we will actually be interested in working with \emph{relative} stable maps, where components contracted into the boundary instead map non-constantly into expansions of the pair $(Y,D)$. 

We will consider two cases. The common data in the two cases is as follows. Let $R$ and $\smash{R'}$ be two indexing sets such that $\smash{ |R| = |R'|+1 }$ and a surjective function $\rho:R \to R'$. Let $\tilde{\alpha}'$ be the unique label with two $\rho$-preimages $\smash{ \tilde{\alpha}_1 }$ and $\smash{ \tilde{\alpha}_2 }$, so that $\rho$ induces a bijection $\smash{ R / (\tilde{\alpha}_1 \sim \tilde{\alpha}_2) \cong R' }$. We will consider two profiles $\smash{ \Gamma }$ and $\smash{ \Gamma' }$ with the same number $n$ of legs, two functions $\lambda:R \to V(\Gamma)$ and $\smash{ \lambda': R' \to V(\Gamma')} $ prescribing the roots on each connected component. The diagram
\begin{center}
\begin{tikzpicture}
\matrix [column sep  = 20mm, row sep = 7mm] {
	\node (nw) {$R$}; &
	\node (ne) {$R'$};  \\
	\node (sw) {$V(\Gamma)$}; &
	\node (se) {$V(\Gamma')$}; \\
};
\draw[->, thin] (nw) -- (ne);
\draw[->, thin] (ne) -- (se);
\draw[->, thin] (nw) -- (sw);
\draw[->, thin] (sw) -- (se);

\node at (0,-0.4) {$\rho_V$};
\node at (0,0.8) {$\rho$};
\node at (-1.8,0.1) {$\lambda$};
\node at (1.3,0.1) {$\lambda'$};

\end{tikzpicture}
\end{center}
will be called the \emph{root-vertex square}. The multiplicities are related by the formula $\smash{ \mu'(\alpha') = \sum_{\rho(\alpha) = \alpha'} \mu(\alpha) }$, for all $\smash{ \alpha' \in R' }$. Let $(p_{\alpha'})_{\alpha' \in R'} \in D(\kk)$ a collection of points on $D$ indexed by $R'$. Let $p_\alpha = p_{\rho(\alpha)}$ for $\alpha \in R$, so that $\smash{ p_{\tilde{\alpha}_1} = p_{\tilde{\alpha}_2} }$. Additionally, we have genus functions $g, g'$ with values in ${\mathbb N}$.

In the first case, we assume that $g(v) = g'(\rho(v))$, $\smash{ d \equiv d' }$, and that the root-vertex square is commutative, with $\rho_V$ a bijection. 

\begin{defn}\label{divergent doppelganger}
Let $N \geq 0$ be an integer. With the notation and assumptions above, a divergent doppelg\"{a}nger pair is a pair of relative stable maps 
\begin{equation*}
\begin{cases}
\left( C,(x_i)_{i=\overline{1,n}},(q_{\alpha})_{\alpha \in R},f \right) \in \overline{\modu M}({\mathfrak Y}^\mathrm{rel},\Gamma|\Lambda)(\kk)  \\
\left( C',(x'_i)_{i=\overline{1,n}},(q'_{\alpha'})_{\alpha' \in R'},f' \right) \in \overline{\modu M}({\mathfrak Y}^\mathrm{rel},\Gamma'|\Lambda')(\kk)
\end{cases}
\end{equation*}
such that the following conditions are satisfied:

(1) $f$ maps into the $(N+1)$-st order expansion $\smash{ {\mathfrak Y}^\mathrm{rel}[N+1] }$, while $\smash{ f' }$ maps into the previous expansion $\smash{ {\mathfrak Y}^\mathrm{rel}[N] }$ such that $ (\eta_{N+1} \circ f)(q_{\alpha}) = (\eta_N \circ f')(q'_{\alpha'}) = p_{\rho(\alpha)}$ for all $\alpha \in R$;

(2) $\eta_{N+1,N}(f(x_i)) = f'(x'_i)$, for all $i=1,2,...,n$; and

(3) There exists an isomorphism $\gamma:C \to C' \cup \bigcup_{\alpha' \in R'} C_{\alpha'}$, with $C_{\alpha'}$ a smooth rational curve attached transversally to $C'$ at the distinguished marked point $q_{\alpha'}$, such that the diagram
\begin{center}
\begin{tikzpicture}
\matrix [column sep  = 17mm, row sep = 8mm] {
	\node (nw) {$C'$}; &
	\node (ne) {$C$};  \\
	\node (sw) {${\mathfrak Y}^\mathrm{rel}[N]$}; &
	\node (se) {${\mathfrak Y}^\mathrm{rel}[N+1]$}; \\
};
\draw[->, thin] (nw) -- (ne);
\draw[right hook->, thin] (sw) -- (se);
\draw[->, thin] (nw) -- (sw);
\draw[->, thin] (ne) -- (se);

\node at (0,1) {$\gamma^{-1}|_{C'}$};
\node at (-2.1,0.1) {$f'$};
\node at (1.3,0.1) {$f$};
\end{tikzpicture}
\end{center} 
is commutative and the restriction of $f$ to $\gamma^{-1}(C_{\alpha'})$ is a degree $\mu'(\alpha')$ cover of a fiber of the ruling of the final ruled component of the expansion ${\mathfrak Y}^\mathrm{rel}[N+1]$, with multiplicity $\mu(\alpha)$ at $q_{\alpha}$, and total ramification at $\gamma^{-1}(q'_{\rho(\alpha')})$.

The integer $N$ is called the depth of the divergent doppelg\"{a}nger.
\end{defn}

Now we consider the second scenario. In this case, we require the root-vertex diagram to be a pushout in the category of sets and the genus functions satisfying conditions (\ref{genus convergent connected}) and (\ref{genus convergent disconnected}) below. Either $|V(\Gamma')| = |V(\Gamma)|$ or $|V(\Gamma')| = |V(\Gamma)| - 1$. In the former case, we require that
\begin{equation}\label{genus convergent connected}
g'(\rho_V(v)) =
\begin{cases}
g(v)+1  & \text{if $v \mapsto \tilde{\alpha}_1$ or $\tilde{\alpha}_2$,} \\[2ex]
g(v) & \text{otherwise},
\end{cases}
\end{equation}
while in the latter case we require that
\begin{equation}\label{genus convergent disconnected}
g'(v') = \sum_{\rho_V(v) = v'} g(v).
\end{equation}

\begin{defn}\label{convergent doppelganger}
A convergent doppelg\"{a}nger is a pair of relative stable maps 
\begin{equation*}
\begin{cases}
\left( C,(x_i)_{i=\overline{1,n}},(q_{\alpha})_{\alpha \in R},f \right) \in \overline{\modu M}({\mathfrak Y}^\mathrm{rel},\Gamma|\Lambda)(\kk)  \\
\left( C',(x'_i)_{i=\overline{1,n}},(q'_{\alpha'})_{\alpha' \in R'},f' \right) \in \overline{\modu M}({\mathfrak Y}^\mathrm{rel},\Gamma'|\Lambda')(\kk)
\end{cases}
\end{equation*}
such that the following conditions are satisfied:

(1) $f$ maps into the $N$th order expansion $\smash{ {\mathfrak Y}^\mathrm{rel}[N] }$, while $\smash{ f' }$ maps into the next expansion $\smash{ {\mathfrak Y}^\mathrm{rel}[N+1] }$ such that $(\eta_N \circ f)(q_{\alpha}) = (\eta_{N+1} \circ f')(q'_{\alpha'}) = p_{\rho(\alpha)}$ for all $\alpha \in R$;

(2) $f(x_i) = \eta_{N+1,N}(f'(x'_i))$, for all $i=1,2,...,n$; and

(3) There exists an isomorphism $\gamma:C' \to C \cup \bigcup_{\alpha' \in R'} C_{\alpha'}$, with $C_{\alpha'}$ a smooth rational curve attached transversally to $C$ at all the distinguished marked points whose label $\alpha$ satisfies $\rho(\alpha) = \alpha'$, such that the diagram
\begin{center}
\begin{tikzpicture}
\matrix [column sep  = 17mm, row sep = 8mm] {
	\node (nw) {$C$}; &
	\node (ne) {$C'$};  \\
	\node (sw) {${\mathfrak Y}^\mathrm{rel}[N]$}; &
	\node (se) {${\mathfrak Y}^\mathrm{rel}[N+1]$}; \\
};
\draw[->, thin] (nw) -- (ne);
\draw[right hook->, thin] (sw) -- (se);
\draw[->, thin] (nw) -- (sw);
\draw[->, thin] (ne) -- (se);

\node at (0,1) {$\gamma^{-1}|_{C}$};
\node at (-2.1,0.1) {$f$};
\node at (1.3,0.1) {$f'$};
\end{tikzpicture}
\end{center} 
is commutative and the restriction of $f'$ to $\gamma^{-1}(C_{\alpha'})$ is a degree $\mu'(\alpha')$ cover of a fiber of the ruling of the final ruled component of the expansion ${\mathfrak Y}^\mathrm{rel}[N+1]$, with multiplicity $\mu(\alpha)$ at each $\gamma^{-1}(q_{\alpha})$, $\rho(\alpha) = \alpha'$, and total ramification at $q'_{\rho(\alpha')}$.

The integer $N$ is called the depth of the convergent doppelg\"{a}nger pair.
\end{defn}

\begin{rmk}\label{branch point}
By the Riemann--Hurwitz formula, there is one more simple branch point of the cover $\smash{ f'|_{\gamma^{-1}(C_{\tilde{\alpha}'})} }$ of ${\mathbb P}^1$ to be accounted for. Indeed,
$$ \underbrace{ \text{ram. at $q'_{\rho(\tilde{\alpha}')}$ } }_{\mu'(\tilde{\alpha}')-1} + \sum_{i \in \{1,2\} }\underbrace{\text{ram. at $\gamma^{-1}(q_{\tilde{\alpha}_i})$}}_{\mu(\tilde{\alpha}_i)-1} + \text{ other ram.}  = 2\mu'(\tilde{\alpha}') - 2 $$
This branch point may be any point in the respective fiber of the final component of ${\mathfrak Y}^\mathrm{rel}[N+1]$, with the exception of the points of intersection with $D_N$ and $D_{N+1}$. However, the mobility of this point is, so to speak, cancelled out by the $\kk^\times$-action on the respective component of ${\mathfrak Y}^\mathrm{rel}[N+1]$ which fixes $D_N$ and $D_{N+1}$, meaning that different branch points give isomorphic maps. As a sanity check, a straightforward calculation shows that the cross-ratio of the $4$ ramification points is
$$ \left( q'_{\rho(\tilde{\alpha}')},\text{fourth ramification point};\gamma^{-1}(q_{\tilde{\alpha}_1}),\gamma^{-1}(q_{\tilde{\alpha}_2}) \right) = -\frac{\mu(\alpha_1)}{\mu(\alpha_2)}, $$
so the 4 ramification points don't "have moduli" in conformity with our wishes and contrary to the a priori expectation. This remark applies equally to Definition \ref{divergent doppelganger} if we replace $\smash{ f'|_{\gamma^{-1}(C_{\tilde{\alpha}'})} }$ with the appropriate piece of the relative stable map.
\end{rmk}

\begin{tiny}
\begin{center}
\begin{tikzpicture}[scale=1.2]

\def\aa{2}

\draw [thick, draw=black, fill=gray, opacity=0.1]
	(-7,-1.25) to [bend right] (-4,-1.25) to [bend right] (-4,1.25) to [bend right] (-7,1.25) to [bend right] cycle;

\draw [thick, draw=black, fill=gray, opacity=0.1]
	(-5,-1) to [bend left] (-5,0) to [bend right] (-5,1) -- (-2,1.5) to [bend left] (-2,0) to [bend right] (-2,-1.5) -- cycle;

\draw [thick, draw=black, fill=gray, opacity=0.1]
	(-5+\aa,-1) to [bend left] (-5+\aa,0) to [bend right] (-5+\aa,1) -- (-2+\aa,1.5) to [bend left] (-2+\aa,0) to [bend right] (-2+\aa,-1.5) -- cycle;

\draw [thick, draw=black, fill=gray, opacity=0.1]
	(-5+2*\aa,-1) to [bend left] (-5+2*\aa,0) to [bend right] (-5+2*\aa,1) -- (-2+2*\aa,1.5) to [bend left] (-2+2*\aa,0) to [bend right] (-2+2*\aa,-1.5) -- cycle;
	
\draw (-5,-1) to [bend left, very thick] (-5,0) to [bend right, very thick] (-5,1);

\draw (-5+\aa,-1) to [bend left, very thick] (-5+\aa,0) to [bend right, very thick] (-5+\aa,1);

\draw (-5+2*\aa,-1) to [bend left, very thick] (-5+2*\aa,0) to [bend right, very thick] (-5+2*\aa,1);

\draw (-5+3*\aa,-1) to [bend left, very thick] (-5+3*\aa,0) to [bend right, very thick] (-5+3*\aa,1);

\draw [blue, very thick] (-0.88,0.7) -- (1.12,0.7);

\draw plot [draw=blue, mark = *, mark size = 1] (0.3, 0.7);

\node at (-0.26,0.5) {$f'(C_{\tilde{\alpha}'})$};

\node at (0.3,-0.9) {$f'(C_{\alpha'})$};

\draw [->, thin] (0.4, 0.8) -- (0.7,0.8);
\draw [->, thin] (0.2, 0.8) -- (-0.1,0.8);  

\def\p{-0.88}
\def\h{4.5}

\draw [blue, very thick] (\p,\h) to [in = 180, out = 90] (\p+0.3,\h+0.3);
\draw [blue, very thick] (\p,\h) to [in = 90, out = -90] (\p+0.3,\h-0.3);

\draw [blue, very thick] (\p,\h-0.6) to [in = -90, out = 90] (\p+0.3,\h-0.3);
\draw [blue, very thick] (\p,\h-0.6) to [in = 180, out = -90] (\p+0.3,\h-0.9);

\draw [blue, very thick] (\p+\aa,\h-0.15) to [in = 0, out = 90] (\p+\aa-0.3,\h+0.3);
\draw [blue, very thick] (\p+\aa,\h-0.15) to [in = 0, out = -90] (\p+\aa-0.3,\h-0.9);

\draw [blue, very thick] (\p + 0.8,\h - 0.6) -- (\p + 1.2,\h - 0.6);
\draw [blue, very thick] (\p + 0.8,\h - 0.9) -- (\p + 1.2,\h - 0.9);

\draw [blue, very thick] (\p+\aa/2,\h) to [in = 180, out = 90] (\p+\aa/2+0.2,\h+0.3);
\draw [blue, very thick] (\p+\aa/2,\h) to [in = 180, out = -90] (\p+\aa/2+0.2,\h-0.3);

\draw [blue, very thick] (\p+\aa/2,\h) to [in = 0, out = 90] (\p+\aa/2-0.2,\h+0.3);
\draw [blue, very thick] (\p+\aa/2,\h) to [in = 0, out = -90] (\p+\aa/2-0.2,\h-0.3);

\draw plot [draw=blue, mark = *, mark size = 1] (\p+1, \h);

\draw [dashed, draw= black, fill=none , opacity=1] (\p+1,\h) circle [radius=0.24];

\node at (\p+1,\h+0.7) {simple ramification};
\node at (\p+1,\h+0.5) {point};

\node at (\p + 0.55,\h+0.2) {...};
\node at (\p + 0.55,\h-0.3) {...};
\node at (\p + 0.55,\h-0.8) {...};

\node at (\p + 1.45,\h+0.2) {...};
\node at (\p + 1.45,\h-0.3) {...};
\node at (\p + 1.45,\h-0.8) {...};

\draw [->, thin] (-2,2.8) to (-2,1.6);
\node at (-2.2,2.2) {$f'$};

\node at (\p - 0.5,\h) {$\gamma^{-1}(q_{\tilde{\alpha}_1})$};

\node at (\p - 0.5,\h-0.6) {$\gamma^{-1}(q_{\tilde{\alpha}_2})$};

\node at (\p +\aa+0.4,\h-0.3) {$q'_{\rho(\tilde{\alpha}')}$};

\node at (\p +\aa+0.2,\h-0.9) {$C_{\tilde{\alpha}'}$};

\node at (1.4,0) {$D_{N+1}$};

\node at (-1.3,0) {$D_N$};

\node at (-5.5,0) {$D=D_0$};

\node at (-6.7,-1) {$Y$};

\draw (-1.12,-0.7) -- (0.88,-0.7);

\def\hh{3}

\draw (\p,\hh-0.15) to [in = 180, out = 90] (\p+0.3,\hh+0.3);
\draw (\p,\hh-0.15) to [in = 180, out = -90] (\p+0.3,\hh-0.9);

\draw (\p+\aa,\hh-0.15) to [in = 0, out = 90] (\p+\aa-0.3,\hh+0.3);
\draw (\p+\aa,\hh-0.15) to [in = 0, out = -90] (\p+\aa-0.3,\hh-0.9);

\draw (\p + 0.8,\hh - 0.9) -- (\p + 1.2,\hh - 0.9);

\draw (\p + 0.8,\hh - 0.5) -- (\p + 1.2,\hh - 0.5);

\draw (\p + 0.8,\hh - 0.1) -- (\p + 1.2,\hh - 0.1);

\draw (\p + 0.8,\hh +0.3) -- (\p + 1.2,\hh +0.3);

\node at (\p + 0.55,\hh+0.2) {...};
\node at (\p + 0.55,\hh-0.3) {...};
\node at (\p + 0.55,\hh-0.8) {...};

\node at (\p + 1.45,\hh+0.2) {...};
\node at (\p + 1.45,\hh-0.3) {...};
\node at (\p + 1.45,\hh-0.8) {...};

\node at (\p - 0.43,\hh-0.3) {$\gamma^{-1}(q_{\alpha})$};

\node at (\p +\aa+0.4,\hh-0.3) {$q'_{\rho(\alpha')}$};

\node at (\p +\aa+0.2,\hh-0.9) {$C_{\alpha'}$};

\node at (-3.5,3.6) {irreducible components};
\node at (-3.5,3.4) {of the source of $f$};

\end{tikzpicture}
\bigskip

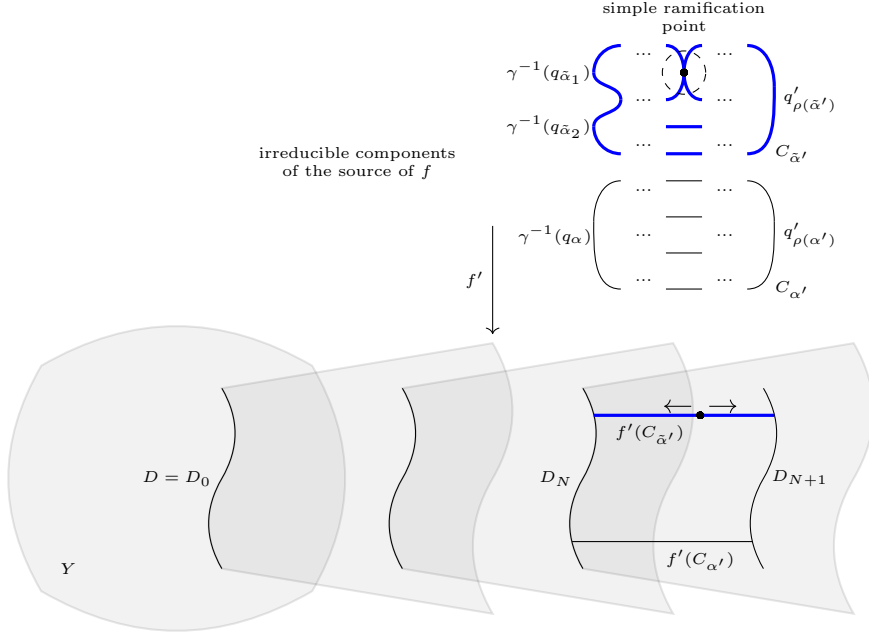
\textbf{Fig. A.} A partial graphical representation of a convergent doppelg\"{a}nger.
\end{center}
\end{tiny}

Now it is a good time to explain the way in which we will utilize the notion of doppelg\"{a}nger. Consider the general situation studied in ~\cite{[Li01], [Li02]}, i.e. a degeneration $W \to C$ such that $W_0 = Y_1 \cup Y_2$, with $Y_1$, $Y_2$ smooth and connected intersecting transversally along a smooth divisor $D$. Let $(f_1,f'_1)$ be a convergent doppelg\"{a}nger pair for $(Y_1,D)$ and $(f_2,f'_2)$ a divergent doppelg\"{a}nger pair for $(Y_2,D)$ such that the images of the distinguished marked points on $D$ match in the obvious sense. We may glue them in two different ways:
\begin{equation}\label{gluing}
f= f_1 \sqcup f_2 \text{ and } f' = f'_1 \sqcup f'_2.
\end{equation}
Then $\gamma_1^{-1} \sqcup \gamma_2$ gives an isomorphism $ \smash{ f \cong f' } $, so we can regard them as the same map. The point is that, by construction $f$ and $f'$ belong to different "virtual components" of the degeneration of the moduli spaces of stable maps, so we've exhibited a point of intersection of these two components. Such intersections will play a crucial role in our argument.

We conclude this subsection by introducing some ad hoc terminology which will hopefully streamline the exposition in the following sections. We will often encounter situations
\begin{center}
\begin{tikzpicture}
\matrix [column sep  = 10mm, row sep = 8mm] {
	\node (nw) {$\overline{\modu S}$}; &
	\node (ne) {$\overline{\modu M}$};  &
	\node (nee) {$\mathrm{Spec}(\kk)$}; \\ &
	\node (se) {$T$}; & \\
};
\draw[right hook->, thin] (nw) -- (ne);
\draw[->, thin] (ne) -- (se);
\draw[->, thin] (nw) -- (se);
\draw[->, thin] (nee) -- (ne);
\node at (-0.3,0) {$q$};
\node at (0.4,0.8) {$p$};
\end{tikzpicture}
\end{center} 
in which $\smash{ \overline{\modu M} }$ is typically an extremely ill behaved moduli space, $\overline{\modu S}$ is an irreducible component of $\overline{\modu M}$ whose dimension is equal to the "expected dimension" of $\overline{\modu M}$, $T$ is a smooth projective variety and the restriction of $q$ to $\overline{\modu S}$ is dominant. We think of $\overline{\modu S}$ as the "main component" of $\overline{\modu M}$. We say that $\smash{ \overline{\modu M} }$ has good behavior at $p$ relative to $q$ and $\overline{\modu S}$ if (1) $q$ has the expected relative dimension $\mathrm{vdim} \overline{\modu M} - \dim T$ at $p$ (which in practice always implies in particular that $\overline{\modu M}$ has the expected dimension at $p$); and (2) $\overline{\modu S}$ is the only irreducible component of $\overline{\modu M}$ containing $p$. We will occasionally use a stronger notion. We say that $\smash{ \overline{\modu M} }$ has excellent behavior at $p$ relative to $q$ if $q$ is smooth at $p$, $p$ belongs to $\smash{\overline{\modu S}}$ and $\smash{ \overline{\modu M} }$ has locally the expected dimension at $q$.

All the moduli spaces we will be using admit natural obstruction theories for which the "expected dimension," which can be formally defined at a point as $\dim (\text{tangent space}) - \dim (\text{obstruction space})$, is constant over the space and coincides with what one might guess the dimension should be based on various naive arguments. Moreover, we will use quite frequently and freely the fact that the expected dimension is a lower bound for the actual dimension of any component.

\subsection{Divergent doppelg\"{a}ngers in the plane} In this subsection, we discuss divergent doppelg\"{a}ngers in the plane and observe that they are very easy to construct. The reason is that \emph{divergent} doppelg\"{a}ngers are rather artificial and uninteresting: the expansion and the presence of the new components is a purely technical symptom, namely of the requirement that all distinguished marked points be different. Notation is as in the previous sections and subsections. \marginpar{\textcolor{white}{ \begin {tiny} bad! \end{tiny} }} We do not need any ordinary marked points, so the pair of relative stable maps will be denoted by
\begin{equation}\label{plane dop notation}
\begin{cases}
\left( C,f,(q_\alpha)_{\alpha \in R} \right) \in \overline{\modu M}({\mathfrak P}^\mathrm{rel},\Gamma|\Lambda)(\kk) \\
\left( C',f',(q'_{\alpha'} \right)_{\alpha' \in R'}) \in \overline{\modu M}({\mathfrak P}^\mathrm{rel},\Gamma'|\Lambda')(\kk).
\end{cases}
\end{equation}
From now on, we will usually refer to a relative stable map simply as $f$ or $f'$, omitting all the other data. First, we give a name to doppelg\"{a}ngers which are "as nice as possible."

\begin{defn} \label{tame divergent doppelganger}
We say that a divergent doppelg\"{a}nger pair (Definition \ref{divergent doppelganger}) denoted as in (\ref{plane dop notation}) is \emph{tame} if the following conditions are simultaneously satisfied: 

(1) The height $N$ is equal to $0$;

(2) $f'$ maps birationally onto its image, has no contracted components and the image of $f'$ is at worst nodal and the nodes occur away from $L$;

(3) $\smash{ [f] }$ belongs to $\smash{ \overline{\modu S}({\mathfrak P}^\mathrm{rel},\Gamma)(\kk) }$ and $\smash{ [f'] }$ belongs to $\smash{ \overline{\modu S}({\mathfrak P}^\mathrm{rel},\Gamma')(\kk) }$.
\end{defn}

Note that in the definition above we are not requiring the source $C'$ to be smooth. Although all the statements below remain true with this natural additional requirement, it turns out not to be particularly useful. A similar remark applies to the analogous Definition \ref{tame UC dopelganger}.

\begin{prop}\label{plane divergent doppelganger}
Assume that Theorem \ref{main theorem} is true for all degrees up to and including $d$. Let $\Gamma$ and $\smash{ \Gamma' }$ as above and $\smash{ (p_{\alpha'})_{\alpha' \in R'} \in L }$ general closed points. Then there exists a tame convergent doppelg\"{a}nger pair as in (\ref{plane dop notation}) such that $\smash{ \overline{\modu M}({\mathfrak P}^\mathrm{rel},\Gamma) }$ has good behavior at $\smash{ [f] }$ relative to ${\mathbf q}_{\Gamma}$ and $\smash{ \overline{\modu S}({\mathfrak P}^\mathrm{rel},\Gamma) }$ and $\smash{ \overline{\modu M}({\mathfrak P}^\mathrm{rel},\Gamma') }$ has good behavior at $\smash{ [f'] }$ relative to $\smash{ {\mathbf q}_{\Gamma'}}$ and $\smash{ \overline{\modu S}({\mathfrak P}^\mathrm{rel},\Gamma') }$. 
\end{prop}

\begin{proof}
This is quite straightforward. We choose $\smash{[f'] }$ to be general in $\smash{ \overline{\modu S}({\mathfrak P}^\mathrm{rel},\Gamma')(\kk) }$. The excellent behavior at $\smash{ [f'] }$ relative to $\smash{ {\mathbf q}_{\Gamma'} }$ follows from ~\cite[Proposition 2.2]{[CH98]}  and some straightforward deformation theory. The construction of $\smash{[f] }$ is forced by the third requirement in Definition \ref{divergent doppelganger}. Consider a sufficiently small \'{e}tale neighborhood of $\smash{ [f] }$ inside $\smash{ \overline{\modu M}({\mathfrak P}^\mathrm{rel},\Gamma) }$. By the calculation for $\Gamma'$, the closed subscheme of curves in the family which still map into the expansion ${\mathbb P}^2[1]$ has dimension
$$ \dim_{[f']} \overline{\modu M}({\mathfrak P}^\mathrm{rel},\Gamma') =  \mathrm{vdim}\mathop{ } \overline{\modu M}({\mathfrak P}^\mathrm{rel},\Gamma') = \mathrm{vdim}\mathop{ } \overline{\modu M}({\mathfrak P}^\mathrm{rel},\Gamma) - 1, $$
so generically the maps in the neighborhood will map into ${\mathbb P}^2[0]$. Clearly, they will be nice (Definition \ref{nice curve}). The fact that $\smash{ \overline{\modu M}({\mathfrak P}^\mathrm{rel},\Gamma) }$ has good behavior at $\smash{ [f] }$ relative to ${\mathbf q}_{\Gamma}$ and $\smash{ \overline{\modu S}({\mathfrak P}^\mathrm{rel},\Gamma) }$ follows. 
\end{proof}

\subsection{Convergent doppelg\"{a}ngers in the plane} Unlike divergent doppelg\"{a}ngers which are purely technical, convergent doppelg\"{a}ngers are somewhat subtle due to the change of arithmetic genus. The analogue of Proposition \ref{plane divergent doppelganger} in the convergent case is substantially more involved. Note that if Theorem \ref{main theorem} is true for all degrees up to and including $d$, then $\smash{ \overline{\modu S}({\mathfrak P}^\mathrm{rel},\Gamma) }$ is irreducible. This is automatic from the definition of $\smash{ \overline{\modu S}({\mathfrak P}^\mathrm{rel},\Gamma) }$.

Assume that we're in the situation $|V(\Gamma')| = |V(\Gamma)|$ corresponding to (\ref{genus convergent connected}). In this case, we impose the condition
\begin{equation} \label{UC condition}
\begin{aligned}
g(v) \leq {d(v)-1 \choose 2} - \min \{ \mu(\tilde{\alpha}_1), \mu(\tilde{\alpha}_2) \},
\end{aligned}
\end{equation}
where, of course, $\smash{ v = \lambda(\tilde{\alpha}_1) = \lambda(\tilde{\alpha}_2) }$. However, in the situation $|V(\Gamma')| = |V(\Gamma)| - 1$ corresponding to (\ref{genus convergent disconnected}), there is no additional requirement. For reassurance, note that the condition
$\smash{ g(v_1) + g(v_2) - 1 \leq {d(v_1)+d(v_2)-1 \choose 2} - \min \{ \mu(\tilde{\alpha}_1), \mu(\tilde{\alpha}_2) \} }$
analogous to (\ref{UC condition}) is actually automatically satisfied in this case.

\begin{defn} \label{tame UC dopelganger}
We say that a convergent doppelg\"anger pair (Definition \ref{convergent doppelganger}) denoted as in (\ref{plane dop notation}) is \emph{tame} if the following conditions are satisfied: 

(1) the height $N$ is equal to $0$;

(2) $f$ maps birationally onto its image, has no contracted components and with the possible exception of $\smash{ p_{\tilde{\alpha}_1} } = \smash{ p_{\tilde{\alpha}_2} }$, the image of $f$ is at worst nodal and the nodes occur away from $L$, while the images of the two branches at $\smash{ p_{\tilde{\alpha}_1} } = \smash{ p_{\tilde{\alpha}_2} }$ are smooth;

(3) $\smash{ [f] }$ belongs to $\smash{ \overline{\modu S}({\mathfrak P}^\mathrm{rel},\Gamma)(\kk) }$ and $\smash{ [f'] }$ belongs to $\smash{ \overline{\modu S}({\mathfrak P}^\mathrm{rel},\Gamma')(\kk) }$.

\end{defn}

For instance, if $|V(\Gamma)| = |V(\Gamma')|=1$, $d=3$, $g=1$ and $\mu \equiv 1$, there are no tame convergent doppelg\"anger pairs. The reason is that (\ref{UC condition}) is violated. Indeed, if there were, it is not hard to see that the node of the image at $\smash{ p_{\tilde{\alpha}_1} } = \smash{ p_{\tilde{\alpha}_2} }$ would need to correspond to a node of the source as well, which contradicts the requirement that the distinguished marked points are different.

\begin{prop}\label{UC dopelganger}
Assume that Theorem \ref{main theorem} is true for all degrees up to and including $d$. Let $\Gamma$ and $\smash{ \Gamma' }$ as above and $\smash{ (p'_{\alpha'})_{\alpha' \in R'} \in L }$ general closed points.  Then there exists a tame convergent doppelg\"{a}nger pair denoted as in (\ref{plane dop notation}) such that $\smash{ \overline{\modu M}({\mathfrak P}^\mathrm{rel},\Gamma) }$ has good behavior at $\smash{ [f] }$ relative to ${\mathbf q}_{\Gamma}$ and $\smash{ \overline{\modu M}({\mathfrak P}^\mathrm{rel},\Gamma) }$ and $\smash{ \overline{\modu M}({\mathfrak P}^\mathrm{rel},\Gamma') }$ has good behavior at $\smash{ [f'] }$ relative to ${\mathbf q}_{\Gamma'}$ and $\smash{ \overline{\modu S}({\mathfrak P}^\mathrm{rel},\Gamma') }$.
\end{prop}

\begin{proof} The proof is deformation-theoretic and inspired by Severi's proof of the non-emptiness of the Severi varieties ~\cite{[Se21]}. For notational convenience, we let $\smash{ p_\alpha = p'_{\rho(\alpha)} }$, for all $\alpha \in R$. Recall that there are two situations, corresponding to (\ref{genus convergent connected}) and (\ref{genus convergent disconnected}) respectively. Although they look notationally different, they are ultimately the same, if in (\ref{genus convergent disconnected}) we treat the disjoint union of the two connected components incident to $\tilde{\alpha}_1$ and $\tilde{\alpha}_2$ as a single curve. Thus we will only deal with the case (\ref{genus convergent connected}) in this proof since the other one is analogous. Moreover, note that it trivially suffices to prove the statement in the special case $|V(\Gamma)| = 1$.

For each $\alpha \in R$, let $Y_\alpha$ be a general irreducible degree $\mu(\alpha)$ rational curve in ${\mathbb P}^2$ whose intersection with $L$ is scheme-theoretically $\mu(\alpha)p_\alpha$. By ~\cite[Proposition 2.2]{[CH98]} and Example \ref{hs algebra}, (1) such curves exist and are at worst nodal; (2) the nodes occur away from $L$; and (3) they only intersect transversally at smooth points away from $L$, with the sole exception of $Y_{\tilde{\alpha}_1}$ and $Y_{\tilde{\alpha}_2}$, which both contain $p_{\tilde{\alpha}_1} = p_{\tilde{\alpha}_2} \in L$.

Let $C_\alpha \cong {\mathbb P}^1$ be the normalization of $Y_\alpha$ and $f_\alpha:C_\alpha \to {\mathbb P}^2$ the normalization map. Define a multigraph structure on $R$ by adding an edge for each node of $\smash{ \bigcup_{\alpha \in R} Y_\alpha }$ away from $L$. Nodes of the irreducible curves $Y_\alpha$ correspond to loops in $R$. Condition (\ref{UC condition}) rules out the trivial case $|R| = 2$, $d=2$, $\mu \equiv 1$, so $R$ is always connected. Note that
\begin{equation*}
\begin{aligned}
|E(R)| + (Y_{\tilde{\alpha}_1} \cdot Y_{\tilde{\alpha}_2})_{p_{\tilde{\alpha}_1}} &= \sum_{ \{\alpha_1,\alpha_2\} \in {R \choose 2} } (Y_{\alpha_1} \cdot Y_{\alpha_2}) + \sum_{\alpha \in R} \text{no. of loops at } \alpha \\
&=  \sum_{ \{\alpha_1,\alpha_2\} \in {R \choose 2} } \mu(\alpha_1) \mu(\alpha_2) + \sum_{\alpha \in R} {\mu(\alpha) -1 \choose 2} \\ &= {d-1 \choose 2} + |R| - 1,
\end{aligned}
\end{equation*}
which we may rearrange as
$$ \mathrm{rank} \mathop{} \mathrm{H}^1(R,\zz)  = {d-1 \choose 2} - \min \{ \mu(\tilde{\alpha}_1), \mu(\tilde{\alpha}_2) \} \geq g, $$
by assumption (\ref{UC condition}). Therefore, there exists a subgraph $G$ which is connected and has the property that $\mathrm{rank} \mathop{} \mathrm{H}^1(G,\zz) = g$. We define $C$ by gluing transversally the $C_\alpha$ at the pairs of preimages of nodes of $\smash{ \bigcup_{\alpha \in R} Y_\alpha }$ corresponding to edges of the subgraph $G$. Then $C$ is semistable of arithmetic genus $g$ because all irreducible components are rational. The nodes $r_e$ of $C$ correspond to edges $e \in E(G)$. The map $f$ is defined in the obvious way by gluing the normalization maps $f_\alpha$. Finally, we set the distinguished marked points $q_\alpha = f_\alpha^{-1}(p_\alpha) \in C_\alpha \subset C$.

We claim that the relative stable map $\left( C,f,(q_\alpha)_{\alpha \in R} \right) \in \overline{\modu M}({\mathfrak P}^\mathrm{rel},\Gamma|\Lambda)(\kk)$ we've constructed actually lies in $\smash{ \overline{\modu S}({\mathfrak P}^\mathrm{rel},\Gamma|\Lambda)(\kk) }$ (\ref{defrel}), as desired. The space of first order deformations of the relative stable map keeping the contact points with $L$ fixed can be identified with
\begin{equation}\label{1st order defs}
{\mathrm T}_{[f]} \overline{\modu M}({\mathfrak P}^\mathrm{rel},\Gamma|\Lambda) \cong \mathrm{H}^0\left(C,{\sh N}_f\left(- D \right ) \right),
\end{equation}
where $\smash{ D = \sum_{\alpha \in R} \mu(\alpha) q_\alpha }$ and ${\sh N}_f$ is an invertible sheaf on $C$ characterized by the following property: the restriction of ${\sh N}_f$ to each component $C_\alpha$ of $C$ is
\begin{equation} \label{ElemMod}
{\sh N}_f|_{C_\alpha} \cong {\sh N}_{f_\alpha} \left( N_\alpha \right), 
\end{equation}
where $N_\alpha=  \sum_{e \in E_\alpha(G)} r_e \in \mathrm{Div}(C_\alpha)$ and $E_\alpha(G)$ denotes the set of edges of $G$ incident to the vertex labeled $\alpha$. Note that in view of (\ref{ElemMod}), if $e$ is an edge between vertices indexed $\alpha_1$ and $\alpha_2$, we have
$$ {\sh N}_f \otimes {\kk}_{r_e} \cong {\sh N}_{f_{\alpha_1}} (r_e) \otimes {\kk}_{r_e}  \cong {\sh O}_{C_{\alpha_1}}(r_e) \otimes {\sh N}_{f_{\alpha_1}} \otimes {\kk}_{r_e} \cong {\mathrm T}_{r_e} C_{\alpha_1} \otimes {\mathrm T}_{r_e} C_{\alpha_2}, $$
which is canonically identified with the versal deformation space of the node $r_e$. Therefore, there is a natural map
\begin{equation*}
\nu : {\mathrm T}_{\left[ f \right]}\overline{\modu M}({\mathfrak P}^\mathrm{rel},\Gamma|\Lambda) \longrightarrow \bigoplus_{e \in E(G)} \mathrm{Def}^1(r_e). 
\end{equation*}
The obstruction space can be identified with ${\mathrm H}^1$ of the same sheaf (\ref{1st order defs}). 
 
We claim that the twist of $\smash{ {\sh N}_f }$ governing the deformations and obstructions of the map has vanishing $\mathrm{H}^1$ and thus ${\mathrm H}^0$ of expected dimension and  that the map $\nu$ defined above is surjective. This implies an analogous statement for deformations of $f$ as an element of $\smash{ \overline{\modu M}({\mathfrak P}^\mathrm{rel},\Gamma) }$ because $\smash{\overline{\modu M}({\mathfrak P}^\mathrm{rel},\Gamma|\Lambda)}$ is obtained by taking the vanishing of sections of $|R|$ line bundles and moreover, it also implies that we may find such deformations matching any tangent vectors at the contact points. Alternatively, we could simply repeat the argument below with the twist of $\smash{ {\sh N}_f }$ as in (\ref{stuff2.2}) and the argument goes through equally well. Regardless, once this is done, we simply observe that the general member of a one-dimensional family of relative stable maps in $\smash{ \overline{\modu M}({\mathfrak P}^\mathrm{rel},\Gamma) }$ with general tangent vector at $[f]$ indeed corresponds to a nice curve (Definition \ref{nice curve})\marginpar{\textcolor{white}{ \begin{tiny} careful ... \end{tiny} }} and therefore $[f]$ indeed belongs to $\smash{ \overline{\modu S}({\mathfrak P}^\mathrm{rel},\Gamma|\Lambda) }$, as desired. We have a short exact sequence
\begin{equation}\label{ses1}
0 \longrightarrow {\sh N}_f(-D) \longrightarrow \bigoplus_{\alpha \in R}  {\sh N}_{f_\alpha}(N_\alpha -\mu(\alpha)q_\alpha) \longrightarrow \bigoplus_{e \in E(G)} \mathrm{Def}^1(r_e) \otimes \kk_{r_e} \longrightarrow 0,
\end{equation}
where the middle terms are implicitly pushed forward to $C$, so the space of first order deformations and the obstruction space are respectively the kernel and cokernel of the map
$$ \bigoplus_{\alpha \in R} \mathrm{H}^0 \left(C_\alpha, {\sh N}_{f_\alpha}(N_\alpha -\mu(\alpha)q_\alpha) \right) \longrightarrow \bigoplus_{e \in E(G)} \mathrm{Def}^1(r_e). $$
By construction, $\deg {\sh N}_{f_\alpha}(N_\alpha -\mu(\alpha)q_\alpha) = 2\mu(\alpha)-2+|E_\alpha(G)|$, so the first cohomology of the middle terms in (\ref{ses1}) vanishes trivially indeed. Of course, this map is simply the composition of 
$$ \bigoplus_{\alpha \in R} \mathrm{H}^0 \left(C_\alpha, {\sh N}_{f_\alpha}(N_\alpha -\mu(\alpha)q_\alpha) \right) \stackrel{\oplus \eta_\alpha}{\longrightarrow} \bigoplus_{\alpha \in R} \bigoplus_{e \in E_\alpha(G)} \mathrm{Def}^1(r_e) = \bigoplus_{e \in E(G)} \mathrm{Def}^1(r_e)^{\oplus 2} $$
with the componentwise difference. Note that the kernel of $\eta_\alpha$ has dimension $h^0({\sh N}_{f_\alpha}( -\mu(\alpha)q_\alpha)) = 2\mu(\alpha) - 1$, so $\oplus \eta_\alpha$ is surjective. This implies all three claims at the beginning of this paragraph at once; formally, we consider the diagram
\begin{flushleft}
\begin{tikzpicture}
\matrix [column sep  = 6mm, row sep = 6mm] {
	\node (nww) {$0$}; &
	\node (nw) {$\mathrm{H}^0({\sh N}_f(-\Sigma) )$}; &
	\node (nc) {$\displaystyle{ \bigoplus \mathrm{H}^0 \left(C_\alpha, {\sh N}_{f_\alpha}(N_\alpha -\mu(\alpha)q_\alpha) \right) }$}; &
	\node (ne) {$\displaystyle{ \bigoplus \mathrm{Def}^1(r_e) }$};  &
	\node (nee) {$0$}; \\
	\node (sww) {$0$}; &
	\node (sw) {$\displaystyle{ \bigoplus \mathrm{Def}^1(r_e) }$}; &
	\node (sc) {$\displaystyle{ \bigoplus \mathrm{Def}^1(r_e)^{\oplus 2} }$}; &
	\node (se) {$\displaystyle{ \bigoplus \mathrm{Def}^1(r_e) }$}; &
	\node (see) {$0$}; \\
};


\draw[->, thin] (nww) -- (nw);
\draw[->, thin] (nw) -- (nc);
\draw[->, thin] (nc) -- (ne);
\draw[->, thin] (ne) -- (nee);

\draw[->, thin] (sww) -- (sw);
\draw[->, thin] (sw) -- (sc);
\draw[->, thin] (sc) -- (se);
\draw[->, thin] (se) -- (see);

\draw[->, thin] (nw) -- (sw);
\draw[->, thin] (nc) -- (sc);
\draw[double equal sign distance, thin] (ne) -- (se);

\node at (0.5,0) {$\oplus \eta_\alpha$};
\node at (-4,0) {$\nu$};

\node at (-2.3,-0.4) {$\mathbf{1} \oplus \mathbf{1}$};

\end{tikzpicture}
\end{flushleft} 
Exactness of the top row is ensured by the surjectivity of $\oplus \eta_\alpha$. Applying the snake lemma, we conclude that the left vertical map is surjective, as desired. Since the zeroth cohomology group has the expected dimension, the fact that $\smash{ \overline{\modu M}({\mathfrak P}^\mathrm{rel},\Gamma) }$ has good behavior at $\smash{ \left( C,f,(q_\alpha)_{\alpha \in R} \right) }$ relative to ${\mathbf q}_{\Gamma}$ and $\smash{ \overline{\modu S}({\mathfrak P}^\mathrm{rel},\Gamma) }$ follows immediately.

Now we have to construct the corresponding $f'$ and prove that it has the desired property. The construction of $f'$ is forced by property (3) in the statement of Definition \ref{convergent doppelganger} and Remark \ref{branch point}. To prove that it is indeed a limit of nice curves (Definition \ref{nice curve}), we argue in two steps. 

First, we analyze the $3$-piece component $C'_0 = \gamma^{-1}(C_{\tilde{\alpha}_1} \cup C_{\tilde{\alpha}_2} \cup C_{\tilde{\alpha}'})$ and the restriction $f'_0$ of $f'$ to $C'_0$. We will temporarily abuse notation by suppressing the isomorphism $\gamma$. The topological data $\Gamma'_0$ corresponding to $f'_0$ has genus $0$, degree $\mu(\tilde{\alpha}_1) + \mu(\tilde{\alpha}_2)$ and one root (distinguished marked point) of multiplicity $\mu(\tilde{\alpha}_1) + \mu(\tilde{\alpha}_2)$. Let $\smash{ \overline{\modu M}({\mathfrak P}^\mathrm{rel},\Gamma'_0|p_{\tilde{\alpha}}) }$ be the corresponding space of genus zero relative stable maps with the contact point fixed at $p_{\tilde{\alpha}}$. 

\begin{claim}\label{lazy dimension argument}
The stable map $(C'_0,f'_0,q'_{\tilde{\alpha}'})$ lives inside $\smash{ \overline{\modu S}({\mathfrak P}^\mathrm{rel},\Gamma'_0|p_{\tilde{\alpha}} ) }$.
\end{claim}

\begin{proof}
The local dimension of $\smash{ \overline{\modu M}({\mathfrak P}^\mathrm{rel},\Gamma'_0|p_{\tilde{\alpha}} ) }$ at $(C'_0,f'_0,q'_{\tilde{\alpha}'})$ is at least the expected dimension $2\mu(\tilde{\alpha}_1) + 2\mu(\tilde{\alpha}_2) - 1$. One way to ascertain the truth of this statement goes as follows: the local dimension of $\smash{ \overline{\modu M}({\mathfrak P}^\mathrm{rel},\Gamma'_0) }$ at $(C'_0,f'_0,q'_{\tilde{\alpha}'})$ is at least the expected dimension $2\mu(\tilde{\alpha}_1) + 2\mu(\tilde{\alpha}_2)$, due to the existence of an obstruction theory of the respective expected dimension on the space of relative stable maps ~\cite{[Li02]} and $\smash{ \overline{\modu M}({\mathfrak P}^\mathrm{rel},\Gamma'_0|p_{\tilde{\alpha}} ) }$ is the pullback of the divisor $\smash{ p'_{\tilde{\alpha}'} \in L }$ under the evaluation map at the distinguished marked point, so the dimension can decrease by at most $1$. \'{E}tale locally near $\smash{ (C'_0,f'_0,q'_{\tilde{\alpha}'}) }$, the relative stable maps can only be of two types: either (a) they have the same dual graph and distribution of degrees of components as $f'_0$, mapping into ${\mathbb P}^2[1]$, or (b) they have irreducible source of degree $\smash{ \mu(\tilde{\alpha}_1) + \mu(\tilde{\alpha}_2) }$, mapping into ${\mathbb P}^2[0] = {\mathbb P}^2$. Invoking again ~\cite[Proposition 2.2]{[CH98]} and Remark \ref{branch point}, we see that the ones of the former type only have $\smash{ 2\mu(\tilde{\alpha}_1) + 2\mu(\tilde{\alpha}_2) - 2 }$ moduli, so there must exist deformations of the latter type too. \end{proof}

Now we deform the $3$-piece component separately, while keeping the other components of $f'$ mapping to ${\mathbb P}^2[0]$ fixed. Applying Claim \ref{lazy dimension argument}, we obtain a family of relative stable maps $\smash{ f'_{0,B} }$ over a pointed curve $(B,b)$
\begin{center}
\begin{tikzpicture}
\matrix [column sep  = 15mm, row sep = 6mm] {
	\node (nw) {${\sh C}_0$}; &
	\node (ne) {$ \mathrm{Bl}_{(t) \times L} \left( \mathrm{Spec}({\kk}[t]) \times {\mathbb P}^2 \right) $};  \\
	\node (sw) {$B$}; &
	\node (se) {$\mathrm{Spec}({\kk}[t])$}; \\
};
\draw[->, thin] (nw) -- (ne);
\draw[->, thin] (sw) -- (se);
\draw[->, thin] (nw) -- (sw);
\draw[->, thin] (ne) -- (se);

\node at (-3.1,0) {$\psi$};
\end{tikzpicture}
\end{center}  
(we refer the reader to ~\cite{[Li01]} for the general theory of families of relative stable maps) such that the central degenerate fiber $(f'_{0,B})_b$ is isomorphic to $f'_0$ and all other fibers $(f'_{0,B})_x$ correspond to nice curves relative to the profile $\Gamma'_0$. (We are suppressing the data of the source and the distinguished marked point from the notation.) We can trivially construct a family 
$$ f'_B:{\sh C} \longrightarrow B \times_{\mathrm{Spec}({\kk}[t])} \mathrm{Bl}_{(t) \times L} \left( \mathrm{Spec}({\kk}[t]) \times {\mathbb P}^2 \right) $$
of relative stable maps in $\smash{ \overline{\modu M}({\mathfrak P}^\mathrm{rel},\Gamma'|\Lambda') }$, the source of $(f'_B)_x$ being $\psi^{-1}(x) \cup \bigcup_{\alpha \in R \backslash \{\tilde{\alpha}_1,\tilde{\alpha}_2\}} C_\alpha$ glued appropriately, $x \neq b$, in such a way that the restriction of $\smash{ f'_{B\backslash \{b\} } }$ to components which don't specialize to components of $C'_0$ is constant over $B \backslash \{b\}$ and the $b$-fiber of $\smash{ f'_B}$ is isomorphic to $\smash{ f'}$. Of course, the nodes of the source of $(f'_B)_x$ which also lie on the moving component $\smash{ \psi^{-1}(x) }$ will vary, but this is unimportant. We may have to shrink $B$ to ensure the intersections remain nodal. Then, simply by repeating the argument in the first half of this proof, we can prove that all maps $(f'_B)_x$ belong to $\smash{ \overline{\modu S}({\mathfrak P}^\mathrm{rel},\Gamma') }$ and therefore, so does $[f']$. Tameness (Definition \ref{tame UC dopelganger}) follows by construction.

We need to check that $\smash{ \overline{\modu M}({\mathfrak P}^\mathrm{rel},\Gamma') }$ has good behavior at $\smash{ [f'] }$ relative to ${\mathbf q}'_{\Gamma}$ and $\smash{ \overline{\modu S}({\mathfrak P}^\mathrm{rel},\Gamma') }$. The main idea is that, by the standard dimension calculations, the local dimension at $[f']$ of the closed substack where the target of the relative stable map is ${\mathbb P}^2[1]$ rather than simply ${\mathbb P}^2$ is strictly smaller than $\smash{ \mathrm{vdim} \overline{\modu M}({\mathfrak P}^\mathrm{rel},\Gamma') }$. Hence for any irreducible component of $\smash{ \overline{\modu M}({\mathfrak P}^\mathrm{rel},\Gamma') }$ containing $[f']$, the generic elements map into ${\mathbb P}^2[0]$ and they are actually nice in the sense of Definition \ref{nice curve} relative to the data $\Gamma'$ by \cite[Proposition 2.2]{[CH98]}. We conclude that $\smash{ \overline{\modu S}({\mathfrak P}^\mathrm{rel},\Gamma') }$ is indeed the only irreducible component containing $[f']$. The fact that the local relative dimension of the evaluation map is the expected one is also straightforward. \end{proof}

\subsection{Some elementary doppelg\"{a}ngers in the blown-up plane} We will need analogous statements which we will apply to relative stable maps into the exceptional divisor of the blowup ${\mathbb F}_1$. However, because we will only have to deal with such relative stable maps whose class has intersection number $1$ with the fibers of the ruling of ${\mathbb F}_1$, this is quite straightforward. Let us introduce notation first. Let $L \subset {\mathbb F}_1$ be the $(-1)$-curve and $L_0$ a line disjoint from $L$, i.e. a curve which blows down to a line in the projective plane under the contraction of $L$ and doesn't pass through the point where $L$ is contracted. Clearly, $L_0 \sim L + \text{[fiber]}$ as divisors on ${\mathbb F}_1$. As a completely elementary preliminary observation, note that a divisor whose class has intersection number $1$ with the fibers of the ruling is smooth if and only if no points of intersection with $L_0$ "line up" with points of intersection with $L$ under the obvious isomorphism $L_0 \cong L$. Moreover, if this is the case and we fix the intersection divisors with $L$ and $L_0$, then the space of curves with the given intersections with $L$ and $L_0$ is naturally (a torsor for) $\kk^\times$. 

First, we introduce general notation for dealing with the type of stable maps into ${\mathbb F}_1$ we will be interested in. Fix points $\xi_1,\xi_2, ..., \xi_n \in L_0$ with multiplicities $m_1,m_2,...,m_n$ adding up to $d$. Let $R = R_0 \sqcup R_\kappa$ be an abstract indexing set and a bijective function $\iota: R_0 \to \{1,2,...,|R_0|\}$. Consider closed points $\smash{ (p_{\alpha})_{\alpha \in R} }$ such that $p_{\alpha} \mapsto \xi_{\iota(\alpha)}$ under the natural isomorphism $L \cong L_0$, for all $\alpha \in R_0$. We assign multiplicities $\mu:R \to \nn^*$ such that $\sum_{\alpha \in R} \mu(\alpha) = d-1$ and $\mu(\alpha) = m_{\iota(\alpha)}$, for all $\alpha \in R_0$. 

Let $\Gamma$ contain the following topological data: vertices $V(\Gamma)$ such that the function $\lambda:R \to V(\Gamma)$ whose graph specifies the incidences between vertices and roots is surjective, constant on $R_\kappa$ with value $\kappa \in V(\Gamma)$ and injective on $R_0$ taking values different from $\kappa$, multiplicities $\mu$ as above, degrees chosen suitably such that all $v \neq \kappa$ correspond to maps to a fiber, $g \equiv 0$ and $n$ legs. Let $\smash{ \overline{\modu M}({\mathfrak F}^\mathrm{rel},\Gamma) }$ be the respective moduli space of relative stable maps and $\smash{ \overline{\modu M}({\mathfrak F}^\mathrm{rel},\Gamma|\Lambda,{\mathbf m}) }$ the closed substack parametrizing maps such that $x_i \mapsto \xi_i$ for all $i$ and the order of contact with $L_0$ at $\xi_i$ is $m_i$. Note that $\smash{ {\mathbf q}_{\Gamma}:\overline{\modu M}({\mathfrak F}^\mathrm{rel},\Gamma|\Lambda,{\mathbf m}) \to L^R }$ factors through $L^{R_\kappa} \hookrightarrow L^R$. Let $\smash{ \overline{\modu M}({\mathfrak F}^\mathrm{rel},\Gamma|\Lambda,{\mathbf m};\Pi) }$ be the fiber of ${\mathbf q}_{\Gamma}$ over the tuple $\Pi \in L^R$.

 As in the case of the projective plane, we will be interested in working with the maps which are "as simple as possible" and their deformations. However, while in the case of ${\mathbb P}^2$ the simplest curves were nodal, in the case of ${\mathbb F}_1$, because of the imposed high order contact with $L_0$, even "optimal" generic behavior involves non-reduced curves. Let $\smash{ \modu S({\mathfrak F}^\mathrm{rel},\Gamma|\Lambda,{\mathbf m};\Pi) }$ be the open substack parametrizing relative stable maps $(C,(q_\alpha)_{\alpha \in R},(x_i)_{i=\overline{1,n}},f)$ mapping to the trivial expansion 
 $$ f: C = C_\kappa \cup \bigcup_{\alpha \in R_0} C_\alpha \longrightarrow {\mathbb F}_1 = {\mathbb F}_1[0] $$
 such that $C_\kappa \cong C_\alpha \cong {\mathbb P}^1$, the restriction of $f$ to $C_\alpha$ is a degree $\smash{ \mu(\alpha) = m_{\iota(\alpha)} }$ map from $C_\alpha$ onto the fiber of $\smash{ {\mathbb F}_1 }$ through $p_{\alpha}$ and $\smash{ \xi_{\iota(\alpha)} }$, completely branched at $p_{\alpha}$ and $\smash{ \xi_{\iota(\alpha)} }$ and the restriction of $f$ to $C_\kappa$ maps isomorphically onto a smooth curve in ${\mathbb F}_1$ which intersects each fiber of the ruling of ${\mathbb F}_1$ exactly once. Note that $\smash{ \modu S({\mathfrak F}^\mathrm{rel},\Gamma|\Lambda,{\mathbf m};\Pi) }$ may be empty for special configurations of points $p_{\alpha}$ and $\xi_i$. The stacks $\smash{ \overline{\modu S}({\mathfrak F}^\mathrm{rel},\Gamma|\Lambda,{\mathbf m};\Pi) }$ form a flat family over some open subset of $L^R$. Taking the closure, we obtain a substack $\smash{ \overline{\modu S}({\mathfrak F}^\mathrm{rel},\Gamma|\Lambda,{\mathbf m}) }$ of $\smash{ \overline{\modu M}({\mathfrak F}^\mathrm{rel},\Gamma|\Lambda,{\mathbf m}) }$. Let $\smash{ \overline{\modu S}({\mathfrak F}^\mathrm{rel},\Gamma|\Lambda,{\mathbf m};\Pi) }$ be the fiber or $\smash{ {\mathbf q}_{\Gamma}:\overline{\modu S}({\mathfrak F}^\mathrm{rel},\Gamma|\Lambda,{\mathbf m}) \to L^R }$ over $\Pi \in L^{R}$. 

Now we discuss doppelg\"{a}ngers. First, we consider the divergent case. Consider the relevant data, as in \S2.1: the topological profiles $\Gamma,\Gamma'$, $\smash{ p_{\tilde{\alpha}_1} = p_{\tilde{\alpha}_2} }$ for two given distinct $\smash{ \tilde{\alpha}_1 \neq \tilde{\alpha}_2 \in R_\kappa }$, the map $\smash {\rho:R \to R' \cong R / (\tilde{\alpha}_1 \sim \tilde{\alpha}_2) }$, etc. A divergent doppelg\"{a}nger pair 
\begin{equation} \label{div dop F notation}
\begin{cases}
\left( C,(x_i)_{i=\overline{1,n}},(q_{\alpha})_{\alpha \in R},f \right) \in \overline{\modu M}({\mathfrak F}^\mathrm{rel},\Gamma|\Lambda,{\mathbf m};\Pi)(\kk) \\
\left( C',(x'_i)_{i=\overline{1,n}},(q'_{\alpha'})_{\alpha' \in R'},f' \right) \in \overline{\modu M}({\mathfrak F}^\mathrm{rel},\Gamma'|\Lambda,{\mathbf m};\Pi')(\kk)
\end{cases}
\end{equation}
is said to be \emph{tame} in this case, if $\smash{ \left( C',(x'_i)_{i=\overline{1,n}},(q'_{\alpha'})_{\alpha' \in R'},f' \right) }$ belong to the substack $\smash{ {\modu S}({\mathfrak F}^\mathrm{rel},\Gamma'|\Lambda,{\mathbf m};\Pi') }$ and all the points $p'_{\alpha'}$ are distinct.

\begin{lem}\label{F-dopelganger divergent}
Let $\Gamma$ and $\Gamma'$ as above and $p_{\alpha}$ and $\xi_i$ configurations of points which satisfy the requirements above and are otherwise general. Then there exists a tame divergent doppelg\"{a}nger pair as in (\ref{div dop F notation}) such that $\smash{ \overline{\modu M}({\mathfrak F}^\mathrm{rel},\Gamma|\Lambda,{\mathbf m}) }$ has good behavior at $\smash{ ( C,(q_{\alpha})_{\alpha \in R} ,(x_i)_{i=\overline{1,n}},f )}$ relative to ${\mathbf q}_{\Gamma}$ and $\smash{ \overline{\modu S}({\mathfrak F}^\mathrm{rel},\Gamma|\Lambda,{\mathbf m}) }$ and $\smash{ \overline{\modu M}({\mathfrak F}^\mathrm{rel},\Gamma'|\Lambda,{\mathbf m}) }$ has good behavior at $\smash{ ( C',(q'_{\alpha'})_{\alpha' \in R'},(x'_i)_{i=\overline{1,n}},f' ) }$ relative to $\smash{ {\mathbf q}_{\Gamma'} }$ and $\smash{ \overline{\modu S}({\mathfrak F}^\mathrm{rel},\Gamma'|\Lambda,{\mathbf m}) }$. 
\end{lem}

\begin{proof}
The same type of argument as in Proposition \ref{plane divergent doppelganger} will work. However, the required dimension counts are trivial in this case.
\end{proof}

A similar property holds for convergent doppelg\"{a}ngers. Again, we have data as in \S2.1: the topological profiles $\Gamma,\Gamma'$, the map $\smash {\rho:R \to R' \cong R / (\tilde{\alpha}_1 \sim \tilde{\alpha}_2) }$, $\smash{ p_{\tilde{\alpha}_1} = p_{\tilde{\alpha}_2} }$, for two distinct $\smash{ \tilde{\alpha}_1 \neq \tilde{\alpha}_2 \in R }$, etc. with the main difference that $\smash{ \tilde{\alpha}_1  \in R_\kappa }$ and $\smash{ \tilde{\alpha}_2  \in R_0 }$ in this case. Consider a convergent doppelg\"{a}nger pair (same notation as in (\ref{div dop F notation})). We say that such a pair is tame if $\smash{ \left( C,(x_i)_{i=\overline{1,n}},(q_{\alpha})_{\alpha \in R},f \right) } $ belongs to $\smash{ \overline{\modu S}({\mathfrak F}^\mathrm{rel},\Gamma|\Lambda,{\mathbf m};\Pi) }$, all the points $p'_{\alpha'}$ are distinct and they only "line up" under $L \cong L_0$ with $\smash{ \xi_{\iota(\tilde{\alpha}_2)} }$ among all $\xi_i$.

\begin{lem}\label{F-dopelganger convergent}
Let $\Gamma$ and $\Gamma'$ as above and $p_{\alpha}$ and $\xi_i$ configurations of points which satisfy the requirements above and are otherwise general. Then there exists a tame convergent doppelg\"{a}nger pair denoted as in (\ref{div dop F notation}) so that $\smash{ \overline{\modu M}({\mathfrak F}^\mathrm{rel},\Gamma|\Lambda,{\mathbf m}) }$ has good behavior at $\smash{ ( C,(q_{\alpha})_{\alpha \in R},(x_i)_{i=\overline{1,n}},f ) }$ relative to $\smash{ {\mathbf q}_{\Gamma} }$ and $\smash{ \overline{\modu M}({\mathfrak F}^\mathrm{rel},\Gamma|\Lambda,{\mathbf m}) }$ and $\smash{ \overline{\modu M}({\mathfrak F}^\mathrm{rel},\Gamma'|\Lambda,{\mathbf m}) }$ has good behavior at $\smash{ ( C',(q'_{\alpha'})_{\alpha' \in R'} ,(x'_i)_{i=\overline{1,n}},f' )}$ relative to ${\mathbf q}_{\Gamma'}$ and $\smash{ \overline{\modu S}({\mathfrak F}^\mathrm{rel},\Gamma'|\Lambda,{\mathbf m}) }$.
\end{lem}

\begin{proof}
This is morally similar to Proposition \ref{UC dopelganger}, but much easier, so it is left to the reader. 
\end{proof}

\section{The landscape of topological profiles}

\subsection{The degeneration setup} From now on, we will be concerned with the degeneration of the moduli spaces of maps when the plane undergoes the well-known degeneration to the union of a projective plane with an ${\mathbb F}_1$ surface, following the general theory developed by J. Li ~\cite{[Li01], [Li02]}. In this subsection, we only introduce notation. Let ${\mathbb A}^1 = \mathrm{Spec} (\kk[t])$ and ${\mathbb A}^\times = \mathrm{Spec} (\kk[t]_{(t)})$. 

We have a morphism $\pi:W \to {\mathbb A}^1 $ such that $W$ is smooth, $\pi^{-1} {\mathbb A}^\times \cong {\mathbb P}^2 \times {\mathbb A}^\times$ and $W_0 = P \cup_\ell F$, where $P$ is a projective plane, $F$ is an ${\mathbb F}_1$ surface and $\ell = P \cap F$ is a projective line. Let $L \times {\mathbb A}^1$ be the Zariski closure of $L \times {\mathbb A}^\times$ inside $W$, so that $L_0 \cap \ell = \emptyset$. Let ${\mathfrak P}^\mathrm{rel}$ and ${\mathfrak F}^\mathrm{rel}$ be the stacks of expansions associated to the pairs $(P, \ell)$ and $(F, \ell)$.

We introduce the degeneration of $\smash{ \overline{\modu M}_{g,n}({\mathbb P}^2,d) }$ and of the "stable map-style" compactification of the generalized Severi variety $\smash{ \overline{\modu S}_{g}({\mathbb P}^2,d|\Lambda,{\mathbf m}) \subseteq \overline{\modu M}_{g}({\mathbb P}^2,d|\Lambda,{\mathbf m}) }$. The total space forms a family $\smash{ \overline{\modu M}({\mathfrak W},\Gamma) \to {\mathbb A}^1 }$ which is proper over ${\mathbb A}^1$ by \cite{[Li01]}. The topological data $\Gamma$ in this case only consists of the degree $d$, the genus $g$ and the number $n$ of ordinary marked points.  Whenever $\Gamma_{\mathfrak P}$ and $\Gamma_{\mathfrak F}$ are compatible in the sense of having the same number $r$ of roots and the corresponding roots are weighted identically, the fiber product of the evaluation maps at the distinguished marked points
\begin{equation}\label{distinguished evaluation map}
{\mathbf q}_{\mathfrak P}:\overline{\modu M}({\mathfrak P}^\mathrm{rel},\Gamma_{\mathfrak P}) \longrightarrow \ell^r \text{ and } {\mathbf q}_{\mathfrak F}:\overline{\modu M}({\mathfrak F}^\mathrm{rel},\Gamma_{\mathfrak F}) \longrightarrow \ell^r
\end{equation}
admits a morphism
\begin{equation}\label{gluing map} 
\Phi_{\eta}: \overline{\modu M}({\mathfrak P}^\text{rel},\Gamma_{\mathfrak P}) \times_{\ell^r} \overline{\modu M}({\mathfrak F}^\text{rel},\Gamma_{\mathfrak F}) \longrightarrow \overline{\modu M}({\mathfrak P}^\mathrm{rel} \sqcup {\mathfrak F}^\mathrm{rel},\Gamma_{\mathfrak P} \sqcup \Gamma_{\mathfrak F})
\end{equation}
to a closed substack of $\smash{ \overline{\modu M}({\mathfrak W}_0,\Gamma) }$, which glues two relative stable maps along their distinguished marked points.

\begin{sublem}[{\cite[Proposition 4.13]{[Li01]}}]\label{unimportant Li lemma}
$\Phi_{\eta}$ is proper and \'{e}tale. 
\end{sublem}

 For each compatible pair $(\Gamma_{\mathfrak P}, \Gamma_{\mathfrak F})$, there exists a line bundle ${\mathbf L}_\eta$ on the total space $\smash{ \overline{\modu M}({\mathfrak W},\Gamma) }$ with a global section ${\mathbf s}_\eta$, such that the vanishing locus of ${\mathbf s}_\eta$ is a closed substack $\smash{ \overline{\modu M}({\mathfrak W}_0,\eta) }$ of $\smash{ \overline{\modu M}({\mathfrak W}_0,\eta) }$ which is topologically identical to the image of (\ref{gluing map}):
\begin{equation} 
\left| \overline{\modu M}({\mathfrak W}_0,\eta) \right|_\mathrm{top} \cong \left| \overline{\modu M}({\mathfrak P}^\text{rel} \sqcup {\mathfrak F}^\text{rel},\Gamma_{\mathfrak P} \sqcup \Gamma_{\mathfrak F}) \right|_\mathrm{top}
\end{equation}
For the lack of a better phrase, we will sometimes refer to $\smash{ \overline{\modu M}({\mathfrak W}_0,\eta) }$ as a \emph{virtual component} of $\smash{ \overline{\modu M}({\mathfrak W}_0,\Gamma) }$. Although $\smash{ \overline{\modu M}({\mathfrak W}_0,\Gamma) }$ and $\smash{ \overline{\modu M}({\mathfrak P}^\text{rel} \sqcup {\mathfrak F}^\text{rel},\Gamma_{\mathfrak P} \sqcup \Gamma_{\mathfrak F}) }$ most certainly have different nonreduced structures (which is important in enumerative problems), this subtlety will be irrelevant throughout most of this paper. 

The \emph{topological profile} $\eta$ describing a virtual component $\smash{ \overline{\modu M}({\mathfrak W}_0,\eta) }$ of the central fiber $\smash{ \overline{\modu M}({\mathfrak W}_0,\Gamma) }$ consists of the following data:

$(D_1)$ a connected bipartite multigraph $G$ with vertices $V(G) = V({\mathfrak P}) \sqcup V({\mathfrak F})$ and $r$ \emph{labeled} edges forming the set $E(G)$;

$(D_2)$ a genus function $g:V({\mathfrak P}) \sqcup V({\mathfrak F}) \to \nn$;

$(D_3)$ a weight function $\mu:E(G) \to \nn^*$ at the distinguished marked points;

$(D_4)$ degree functions $d_{\mathfrak P}: V({\mathfrak P}) \to \nn^*$ and $e_{\mathfrak F},d_{\mathfrak F}: V({\mathfrak F}) \to \nn$;\footnote{The treatment of degrees for this particular step in ~\cite{[Li01], [Li02]} is actually \emph{a priori} different. Instead of prescribing the algebraic equivalence classes, one only specifies the degrees relative to some suitable ample line bundles. However, because $\mathrm{Pic}({\mathbb P}^2) = {\mathbb Z}$ and $\mathrm{Pic}({\mathbb F}_1) = {\mathbb Z}^2$, in the case at hand, this data actually pins down the rational equivalence classes on both sides thanks to the kissing condition.}

$(D_5)$ a function $\nu:[n]:=\{1,2,...,n\} \to V(G)$,

\noindent subject to the following list of conditions:

$(C_1)$ $\sum_{v \in V({\mathfrak P}) }d_{\mathfrak P}(v) = \sum_{v \in V({\mathfrak F}) } e_{\mathfrak F}(v) $ and $\sum_{v \in V({\mathfrak F}) }d_{\mathfrak F}(v) = d$;

$(C_2)$ For each vertex $v$
\begin{equation}\label{sum of mult}
\sum_{e \in E(v)} \mu(e) = \begin{cases}
d_{\mathfrak P}(v)  & \text{if $v \in V({\mathfrak P})$,} \\[2ex]
e_{\mathfrak F}(v) & \text{if $v \in V({\mathfrak F})$};
\end{cases}
\end{equation}

$(C_3)$ The genus bound: for all $v \in V({\mathfrak P})$,
\begin{equation}\label{genus bound}
0 \leq g(v) \leq {d_{\mathfrak P}(v)-1 \choose 2};
\end{equation}

$(C_4)$ The topological requirement 
\begin{equation}\label{topological requirement}
\sum_{ v \in V({\mathfrak P})} {g(v)} +  \sum_{v \in V({\mathfrak F})} {g(v)} - |V(G)| + |E(G)| +1 = g.
\end{equation}
Of course, (\ref{topological requirement}) corresponds to the topological requirement that the relative stable maps glue to a map with semistable source of arithmetic genus $g$. We clarify that $e_{\mathfrak F}$ stands for the intersection number with $\ell$, while $d_{\mathfrak F}$ stands for the intersection number with $L_0$. The quantity (the multiplicity of $L$ after degenerating inside ${\mathbb P}^2$)
$$ \sum_{v \in V({\mathfrak F})} \left( d_{\mathfrak F}(v) - e_{\mathfrak F}(v) \right) = d -  \sum_{v \in V({\mathfrak P})} d_{\mathfrak P}(v) $$
will be called the \emph{height} of the given topological profile. 

Now we introduce the degeneration of the generalized Severi varieties. Let $\Lambda$ be the collection of points $p_1,p_2,...,p_n \in L$. As in \S1.2, let $\smash{ \overline{\modu M}_{g}({\mathfrak W},d|\Lambda,{\mathbf m}) }$ be the closed substack of $\smash{ \overline{\modu M}({\mathfrak W},\Gamma) }$ where, informally, the $i$th ordinary marked point maps to $\{p_i\} \times {\mathbb A}^1$ and the the pullback of the section defining $L \times {\mathbb A}^1$ to the source vanishes at least to order $\smash{ m_i }$ at the respective ordinary marked point. Hence its restriction over $\smash{ {\mathbb A}^\times }$ is $\smash{ {\mathbb A}^\times \times \overline{\modu M}_{g}({\mathbb P}^2,d|\Lambda,{\mathbf m}) }$. Let $\smash{ \overline{\modu D}_{g}({\mathfrak W},d|\Lambda,{\mathbf m}) }$ be the closure of $\smash{ {\mathbb A}^\times \times \overline{\modu S}_{g}({\mathbb P}^2,d|\Lambda,{\mathbf m}) }$ endowed with its reduced substack structure.

There is a straightforward parallel discussion concerning the virtual components of this degeneration. For instance, $\smash{ \overline{\modu S}({\mathfrak W}_0,\eta|\Lambda,{\mathbf m}) }$ is the vanishing locus of ${\mathbf s}_\eta$ on $\smash{ \overline{\modu D}_{g}({\mathfrak W},d|\Lambda,{\mathbf m}) }$. However, if $\nu([n]) \cap V({\mathfrak P}) \neq \emptyset$, then $\smash{ \overline{\modu S}({\mathfrak W}_0,\eta|\Lambda,{\mathbf m}) }$ is empty. The analogue of (\ref{gluing map}) will be introduced in \S5 since its use requires some care.

\begin{defn}
The \emph{landscape of topological profiles}, denoted by $\Omega$, is the collection of all topological profiles as above. The \emph{landscape of small topological profiles} $\Omega^s$ is the collection of all \emph{height one} topological profiles, with the additional properties that $\smash{ g|_{V({\mathfrak F})} \equiv 0 }$, $\smash{\nu([n]) = V({\mathfrak F})}$ and $\smash{ \nu^{-1}(v) }$ is a singleton for all $\smash{ v \in V({\mathfrak F}) }$ such that $\smash{ e_{\mathfrak F}(v) = d_{\mathfrak F}(v) }$.
\end{defn}

If $\eta$ is a small topological profile, the height one requirement implies that there exists a unique vertex $v \in V({\mathfrak F})$ such that $d_{\mathfrak F}(v) = e_{\mathfrak F}(v) + 1$ and also that $d_{\mathfrak F}(v) = e_{\mathfrak F}(v)$, for all other $v \in V({\mathfrak F})$. The special vertex will be called the \emph{distinguished $F$-vertex} of $\eta$. Distinguished $F$-vertices will typically be denoted by $\kappa$.

\subsection{The elementary operations on topological profiles} In this subsection, we will define a graph structure on the small landscape $\Omega^s$ defined above. The edges between elements of $\Omega^s$ reflect intersections of virtual components of the degenerate moduli space. Let $\eta$ be a topological profile in $\smash{ \Omega^s }$ with distinguished $F$-vertex $\smash{ \kappa \in V({\mathfrak F}) }$. If $\eta'$ is a modification of $\eta$ of any of the three types below, we draw an edge between $\eta$ and $\smash{\eta'}$, defining an \emph{unoriented} graph structure on $\smash{ \Omega^s }$.

\textit{Upper connected operation.} Given a vertex $v \in V({\mathfrak P})$ and two edges $e_1$ and $e_2$ incident to $v$ and $\kappa$, the topological profile $\eta'$ defined by increasing $g(v)$ by $1$ and replacing $e_1$ and $e_2$ with only one edge $e'$ of multiplicity $\mu'(e') = \mu(e_1) + \mu(e_2)$ is called the upper connected modification of $\eta$ at $v,e_1,e_2$. However, the following requirement needs to be satisfied:
\begin{equation}\label{UC ineq}
\min \{ \mu(e_1),\mu(e_2) \}  \leq {d_{\mathfrak P}(v) - 1 \choose 2 } - g(v)
\end{equation}

\textit{Upper disconnected operation.} Given two vertices $v_1,v_2 \in V({\mathfrak P})$, two edges $e_1$ and $e_2$ between $\kappa$ and $v_1$ respectively $v_2$, the topological profile $\eta'$ defined by replacing the two vertices $v_1$ and $v_2$ with one vertex $v'$ such that $d'_{\mathfrak P}(v') = d_{\mathfrak P}(v_1) + d_{\mathfrak P}(v_2)$, $g'(v') = g(v_1) + g(v_2)$ and the edges $e_1$ and $e_2$ with one edge $e'$ such that $\mu'(e') = \mu(e_1) + \mu(e_2)$ is called the upper disconnected modification of $\eta$ at $v_1,v_2,e_1,e_2$.

\textit{Lower disconnected operation.} Given two vertices $v_p \in V({\mathfrak P})$, $v_f \in V({\mathfrak F}) \backslash \{\kappa\}$ and edges $e$ between $v_p$ and $v_f$ and $\tilde{e}$ between $v_p$ and $\kappa$, the profile $\eta'$ obtained by crimping together $v_f$ and $\kappa$ into a vertex $\kappa'$ with $d'_{\mathfrak F}(\kappa') = d_{\mathfrak F}(\kappa) + d_{\mathfrak F}(v_f)$ and the edges $e$, $\tilde{e}$ into $e'$ with $\mu'(e') = \mu(e_1) + \mu(e_2)$ is called the lower disconnected modification of $\eta$ at $v_p,v_f,e,\tilde{e}$.

\tikzset{circ/.style = {draw, circle, minimum size = 4.5mm},}
\begin{tiny}
\begin{center}

\begin{tikzpicture}[scale=0.6]

\def\xe{-2}
\def\ye{0}

\def\xf{-5}
\def\yf{-6}

\def\xg{1.5}
\def\yg{-6}

\def\xa{8}
\def\ya{-6}

\def\xb{6}
\def\yb{0}

\def\xc{-2}
\def\yc{-11}

\def\xd{6}
\def\yd{-11}

\def\dx{1}
\def\dy{2}

\def\droot{1}
\def\dxroot{0.5}

\node at (\xg,\yg) [circ] (7P1) {1};
\node at (\xg,\yg - \dy) [circ] (7F1) {0}; 

\node at (\xg,\yg - \dy - \droot) (R7F11) {1}; \node at (\xg -\dxroot, \yg - \dy - \droot) (R7F12) {1}; \node at (\xg + \dxroot, \yg - \dy - \droot) (R7F13) {2};

\draw (7P1) -- (7F1);

\draw (7F1) -- (R7F11); 
\draw (7F1) -- (R7F12); 
\draw (7F1) -- (R7F13);

\node at (\xg - 0.2,\yg-1) {3}; 
\node at (\xg + 0.6,\yg-2) {$\kappa$};

\node at (\xf,\yf) [circ] (6P1) {0};
\node at (\xf,\yf - \dy) [circ] (6F1) {0}; 

\node at (\xf,\yf - \dy - \droot) (R6F11) {1}; \node at (\xf -\dxroot, \yf - \dy - \droot) (R6F12) {1}; \node at (\xf + \dxroot, \yf - \dy - \droot) (R6F13) {2};

\draw (6P1) to [in= 120, out = 240] (6F1);
\draw (6P1) to [in= 60, out = 300] (6F1);

\draw (6F1) -- (R6F11); 
\draw (6F1) -- (R6F12); 
\draw (6F1) -- (R6F13);

\node at (\xf - 0.6,\yf-1) {1}; 
\node at (\xf + 0.6,\yf-1) {2}; 
\node at (\xf + 0.6,\yf-2) {$\kappa$};

\node at (\xe,\ye) [circ] (5P1) {0};
\node at (\xe - \dx,\ye - \dy) [circ] (5F1) {0}; \node at (\xe - \dx,\ye - \dy - \droot) (R5F1) {1};
\node at (\xe + \dx,\ye-\dy) [circ] (5F2) {0}; \node at (\xe + \dx -\dxroot, \ye - \dy - \droot) (R5F21) {2}; \node at (\xe + \dx + \dxroot, \ye - \dy - \droot) (R5F22) {1};

\draw (5P1) -- (5F1); 
\draw (5P1) to [in= 140, out = -80] (5F2);
\draw (5P1) to [in= 100, out = -40] (5F2);

\draw (5F1) -- (R5F1); 
\draw (5F2) -- (R5F21); 
\draw (5F2) -- (R5F22);

\node at (\xe - 0.7,\ye-1) {1}; 
\node at (\xe + 0,\ye-1) {1}; 
\node at (\xe + 1,\ye-1) {1}; 
\node at (\xe + 1.6,\ye-2) {$\kappa$};

\node at (\xd,\yd) [circ] (4P1) {1};
\node at (\xd - \dx,\yd - \dy) [circ] (4F1) {0}; \node at (\xd - \dx,\yd - \dy - \droot) (R4F1) {2};
\node at (\xd + \dx,\yd-\dy) [circ] (4F2) {0}; \node at (\xd + \dx -\dxroot, \yd - \dy - \droot) (R4F21) {1}; \node at (\xd + \dx + \dxroot, \yd - \dy - \droot) (R4F22) {1};

\draw (4P1) -- (4F1); 
\draw (4P1) -- (4F2); 

\draw (4F1) -- (R4F1); 
\draw (4F2) -- (R4F21); 
\draw (4F2) -- (R4F22);

\node at (\xd - 0.7,\yd-1) {2}; 
\node at (\xd + 0.7,\yd-1) {1}; 
\node at (\xd + 1.6,\yd-2) {$\kappa$};

\node at (\xc - \dx,\yc) [circ] (3P1) {0};
\node at (\xc + \dx,\yc) [circ] (3P2) {0};
\node at (\xc, \yc - \dy) [circ] (3F1) {0}; 
\node at (\xc - \dxroot,\yc - \dy - \droot) (R3F11) {2};
\node at (\xc, \yc - \dy - \droot) (R3F12) {1};
\node at (\xc + \dxroot, \yc - \dy - \droot) (R3F13) {1};

\draw (3P1) to [in= 140, out = -80] (3F1);
\draw (3P1) to [in= 100, out = -40] (3F1);

\draw (3P2) -- (3F1); 

\draw (3F1) -- (R3F11); 
\draw (3F1) -- (R3F12); 
\draw (3F1) -- (R3F13);

\node at (\xc,\yc-1) {1}; 
\node at (\xc - 1,\yc-1) {1}; 
\node at (\xc + 0.7,\yc-1) {1}; 
\node at (\xc + 0.6,\yc-2) {$\kappa$};

\node at (\xb,\yb) [circ] (2P1) {1};
\node at (\xb - \dx,\yb - \dy) [circ] (2F1) {0}; \node at (\xb - \dx,\yb - \dy - \droot) (R2F1) {1};
\node at (\xb + \dx,\yb-\dy) [circ] (2F2) {0}; \node at (\xb + \dx -\dxroot, \yb - \dy - \droot) (R2F21) {2}; \node at (\xb + \dx + \dxroot, \yb - \dy - \droot) (R2F22) {1};

\draw (2P1) -- (2F1); 
\draw (2P1) -- (2F2); 

\draw (2F1) -- (R2F1); 
\draw (2F2) -- (R2F21); 
\draw (2F2) -- (R2F22);

\node at (\xb - 0.7,\yb-1) {1}; 
\node at (\xb + 0.7,\yb-1) {2}; 
\node at (\xb + 1.6,\yb-2) {$\kappa$};

\node at (\xa,\ya) [circ] (1P1) {1};
\node at (\xa - \dx,\ya - \dy) [circ] (1F1) {0}; \node at (\xa - \dx,\ya - \dy - \droot) (R1F1) {1};
\node at (\xa,\ya-\dy) [circ] (1F2) {0}; \node at (\xa, \ya - \dy - \droot) (R1F2) {1};
\node at (\xa + \dx,\ya-\dy) [circ] (1F3) {0}; \node at (\xa + \dx, \ya - \dy - \droot) (R1F3) {2};

\draw (1P1) -- (1F1); 
\draw (1P1) -- (1F2); 
\draw (1P1) -- (1F3); 

\draw (1F1) -- (R1F1); 
\draw (1F2) -- (R1F2); 
\draw (1F3) -- (R1F3);

\node at (\xa - 0.2,\ya-1) {1}; 
\node at (\xa - 0.7,\ya-1) {1}; 
\node at (\xa + 0.7,\ya-1) {1}; 
\node at (\xa + 1.6,\ya-2) {$\kappa$}; 

\draw [very thick, blue] (\xa,\ya + 0.5) -- (\xb + 0.8,\yb - 3); 

 \node at ( 2,-0.7 ) {conn.};
 \node at ( -1.8,-6.7 ) {conn.};
 \node at ( -3.8,-9.9 ) {disc.};

\draw [very thick, cyan] (\xe + 1.2,\ye - 1) -- (\xb - 1,\yb - 1);
\draw [very thick, blue] (\xe,\ye -3) -- (\xf + 0.5,\yf + 0.5);
\draw [very thick, cyan] (\xc - 1.5,\yc) -- (\xf,\yf - 3.2);
\draw [very thick, cyan] (\xg - 0.5,\yg - 1) -- (\xf + 0.8,\yf - 1);
\draw [very thick, blue] (\xg + 0.5,\yg + 0.5) -- (\xb - 1.2,\yb - 3);
\draw [very thick, blue] (\xg + 0.3,\yg - 1) -- (\xd,\yd + 0.5);

\end{tikzpicture}

\bigskip

\textbf{Fig. B.} Illustration of the case $d=4$, $g=1$, ${\mathbf m} = (2,1,1)$. 
\end{center}
\end{tiny}

\begin{prop}\label{landscape is connected}
The graph structure defined on $\Omega^s$ is connected.
\end{prop}

\begin{proof}
Note that any topological profile in $\Omega^s$ lives in the same connected component as a small topological profile whose associated bipartite graph has only one vertex on either shore. Indeed, if $|V({\mathfrak P})| > 1$ we can always perform an upper disconnected modification to decrease the total number of vertices. Similarly for $|V({\mathfrak F})| > 1$ and lower disconnected modifications. Therefore, it suffices to prove that the induced subgraph on the small topological profiles with $|V({\mathfrak P})| = |V(\mathfrak F)| = 1$ is connected. This problem can be rephrased in terms of partitions of $d-1$. Note that the topological requirement (\ref{topological requirement}) becomes $g(v) + r - 1 = g$, where $v$ is the unique vertex in $V({\mathfrak P})$. Motivated by this and (\ref{genus bound}), let ${\sh P}$ be the set of partitions $p$ of $d-1$ whose length is bounded as follows:
\begin{equation}\label{length interval}
g+1 - {d-2 \choose 2} \leq \mathrm{length}(p) \leq g+1.
\end{equation}
Note that the length corresponds to $r=|E(G)|$. The lower bound is vacuous if $g \leq {d-2 \choose 2}$ and that the upper bound is vacuous if $g \geq d-2$. We introduce three ranges according to these thresholds: $0 \leq g < d-2$ (low genus), $d-2 \leq g \leq {d-2 \choose 2}$ (average genus) and ${d-2 \choose 2} < g \leq {d-1 \choose 2}$ (high genus). Note that
$ {d-2 \choose 2} \geq d-3 $ for all $d$, so these ranges cover all possibilities.

We connect a partition $p$ containing two terms $x$ and $y$ to the partition $p'$ where $x$ and $y$ are replaced by $x+y$ if the following condition corresponding to (\ref{UC condition}) is satisfied:
\begin{equation}\label{min condition}
\min \{ x,y \} \leq {d-2 \choose 2} - g + \mathrm{length}(p) - 1.
\end{equation}
The subgraph above is isomorphic to the graph structure we've defined on ${\sh P}$. Thus we need to prove that the latter is connected.

\begin{center}
\begin{tikzpicture}[scale=0.45]

\def\a{11}
\def\b{5}
\def\c{-2}
\def\cc{-2.5}
\def\d{-2}
\def\e{3}
\def\f{-10}

\draw [thick, draw=black, fill=gray, opacity=0.5]
	(0-\a,2) -- (1-\a,2) -- (1-\a,-2) -- (0-\a,-2) -- cycle;
	
\draw [thin, draw = black] (0-\a,1)--(1-\a,1);
\draw [thin, draw = black] (0-\a,0)--(1-\a,0);
\draw [thin, draw = black] (0-\a,-1)--(1-\a,-1);

\draw [thick, draw=black, fill=gray, opacity=0.5]
	(0-\b,1.5) -- (2-\b,1.5) -- (2-\b,0.5) -- (1-\b,0.5) -- (1-\b,-1.5) -- (0-\b,-1.5) -- cycle;
	
\draw [thin, draw = black] (1-\b,1.5)--(1-\b,0.5);
\draw [thin, draw = black] (0-\b,0.5)--(1-\b,0.5);
\draw [thin, draw = black] (0-\b,-0.5)--(1-\b,-0.5);
	
\draw [thick, draw=black, fill=gray, opacity = 0.5]
	(0-\c,1.5-\d) -- (3-\c,1.5-\d) -- (3-\c,0.5-\d) -- (1-\c,0.5-\d) -- (1-\c,-0.5-\d) -- (0-\c,-0.5-\d) -- cycle;
	
\draw [thin, draw = black] (1-\c,1.5-\d)--(1-\c,0.5-\d);
\draw [thin, draw = black] (0-\c,0.5-\d)--(1-\c,0.5-\d);
\draw [thin, draw = black] (2-\c,1.5-\d)--(2-\c,0.5-\d);

\draw [thick, draw=black, fill=gray, opacity=0.5]
	(0-\cc,2-\e) -- (2-\cc,2-\e) -- (2-\cc,0-\e) -- (0-\cc,0-\e) -- cycle;
	
\draw [thin, draw = black] (0-\cc,1-\e)--(2-\cc,1-\e);
\draw [thin, draw = black] (1-\cc,0-\e)--(1-\cc,2-\e);

\draw [thick, draw=black, fill=gray, opacity=0.5]
	(0-\f,1) -- (4-\f,1) -- (4-\f,0) -- (0-\f,0) -- cycle;
	
\draw [thin, draw = black] (1-\f,0)--(1-\f,1);
\draw [thin, draw = black] (2-\f,0)--(2-\f,1);
\draw [thin, draw = black] (3-\f,0)--(3-\f,1);

\draw [->, very thick, blue] (-9.5,0) -- (-5.5,0);
\draw [->, very thick, blue] (-2.5,1) to [bend left] (1.5,2.5);
\draw [->, very thick, blue] (-3.5,-1) to [bend right] (2,-2);
\draw [->, very thick, blue] (5.5,3) to [bend left] (12,1.5);

\draw [-, thin, dashed, black] (5,-2) to [bend right] (12,-0.5);

\node at (7.5, -1.8) {\begin{tiny}not an edge (\ref{min condition})\end{tiny}};

\node at (-10.5, -3.5) {\begin{tiny} $g(v)=0$ \end{tiny}};

\node at (-4, -3.5) {\begin{tiny} $g(v)=1$ \end{tiny}};

\node at (3.5, -3.5) {\begin{tiny} $g(v)=2$ \end{tiny}};

\node at (12, -3.5) {\begin{tiny} $g(v)=3$ \end{tiny}};

\end{tikzpicture}
\end{center} 

\begin{center}
\begin{tiny}
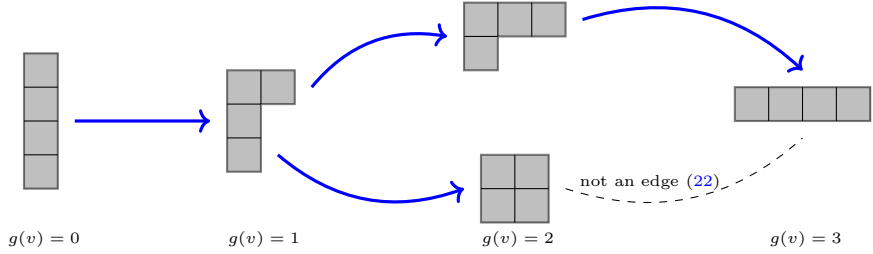
\textbf{Fig. C.} The graph ${\sh P}$ of partitions of $d-1=4$, for $d=5$, $g=3$.
\end{tiny}
\end{center}

\begin{sublem}\label{partitions1}
Let $k \leq n$ be integers, $n$ positive. Let ${\sh P}_{n}^{\geq k}$ be the set of partitions of $n$ of length at least $k$. We connect a partition $p \in {\sh P}_{n}^{\geq k}$ containing two terms $x$ and $y$ to the partition $p' \in {\sh P}_{n}^{\geq k}$ where $x$ and $y$ are replaced by $x+y$, if $\min\{x,y\} + k \leq \mathrm{length}(p)$. Then the graph ${\sh P}_{n}^{\geq k}$ is connected.
\end{sublem}

\begin{proof}
Any partition $p' \in {\sh P}_{n}^{\geq k}$ except $n=1+1+...+1$ contains a term $z \geq 2$, which can be split $z = 1+(z-1)$, so it can be obtained from a partition $p$ of length one bigger in a legal way. Thus in any connected component, the partition of maximal length is $n=1+1+...+1$, so there is a unique connected component.
\end{proof}

\begin{sublem}\label{partitions2}
Let $k < n$ be integers, $n \geq 3$. Let ${\sh P}_{n}^{\leq k}$ be the set of partitions of $n$ of length at most $k$. We connect a partition $p \in {\sh P}_{n}^{\leq k}$ containing two terms $x$ and $y$ to the partition $p' \in {\sh P}_{n}^{\leq k}$ where $x$ and $y$ are replaced by $x+y$. Then the graph ${\sh P}_{n}^{\leq k}$ is connected.
\end{sublem}

\begin{proof}
This is entirely trivial.
\end{proof}

Fact \ref{partitions1} applied for $n=d-1$ and $k = g+1 - {d-2 \choose 2}$ proves the desired connectivity result in the high genus range and, in fact, also in the average genus range because we're not requiring $k$ to be positive. We claim that Fact \ref{partitions2} applied for $n=d-1$, $k = g+1$ implies the result in the low genus range. The point is that (\ref{min condition}), which reads
$$ \min \{ x,y \} \leq {n-1 \choose 2} - k + \mathrm{length}(p), $$
is automatically satisfied because $\lfloor \frac{n}{2} \rfloor \leq {n-1 \choose 2}$ for $ n \geq 3$. Finally, note that the case $d=3$ is fine because (\ref{length interval}) doesn't allow the partitions $2=1+1$ and $2=2$ simultaneously. \end{proof}

\section{Preliminaries: the Caporaso--Harris Theorem revisited} \marginpar{\textcolor{white}{ \begin {tiny} terrible, but maybe unavoidable \end{tiny} }}The Caporaso-Harris analysis ~\cite{[CH98]} was phrased in terms of plane curves which are forced to split off a copy of $L$ by imposing a $(d+1)$-st point of intersection with this line. However, there is an essentially equivalent way to phrase the arguments and results by instead considering the degeneration $\pi:W \to {\mathbb A}^1$ and the specialization of the complete linear systems $|{\sh O}_{W_t}(d)|$, $t \neq 0$, to the complete linear system on $W_0 = P \cup_\ell F$ obtained by gluing the line bundles ${\sh O}_P(d-1)$ and $\delta^*{\sh O}(d) \otimes {\sh O}_F(-(d-1)\ell)$ along $\ell$, where $\delta$ denotes the blowdown of $\ell$ on $F$.

In this language, the main non-technical difference between ~\cite{[CH98]} and the application of ~\cite{[Li01], [Li02]} to the case at hand is that the former (indirectly) utilizes a specific distribution of degrees on the components of the central fiber, while the latter considers all possible distributions of degrees on $P$ and $F$. We want the best of both worlds: we want the fewer components and the simplicity on the $F$-side of the former and at the same time we want the good deformation theoretic behavior of the latter. To do that, in this section we will restate some parts of the Caporaso--Harris analysis in our language. This will later be used to extract Corollary \ref{CH corollary}, which essentially allows us to use the formalism of degenerate stable maps, but only having to worry about the small topological profiles. We emphasize that the sole logical purpose of this section is to prove the respective corollary. 

Consider the pushforward ${\sh E} = \pi_* \left( {\mathrm {proj}}_{{\mathbb P}^2} {\sh O}_{{\mathbb P}^2} (d) \otimes {\sh O}_W(-F) \right)$, which is a rank $\smash{ {d+2 \choose 2} }$ free sheaf over ${\mathbb A}^1$, and its classical projectivization ${\mathbb P}{\sh E} = \mathrm{Proj}_{{\mathbb A}^1} \mathrm{Sym} ({\sh E}^\vee)$. We will use the results of ~\cite[\S3]{[CH98]} to describe the closure of $\smash{ {\mathbb A}^\times \times V_{d,g}(\Lambda,{\mathbf m}) }$ inside ${\mathbb P}{\sh E}$. These ideas are by now classical, so we will be very brief. 

Denote the closure by $\smash{ \overline{V}_{d,g,W}(\Lambda,{\mathbf m}) }$ and let $\smash{ \overline{V}_{d,g,W_0}(\Lambda,{\mathbf m}) }$ be its intersection with ${\mathbb P}{\sh E}_0$. By construction, the latter space is purely of the expected dimension $2d+g-1$. For each topological profile $\eta \in \Omega^s$ in the small landscape, it is straightforward to construct the classical counterpart $\smash{ \overline{V}_\eta \subset {\mathbb P}{\sh E}_0 }$ of the virtual component $\smash{ \overline{\modu S}({\mathfrak W}_0,\eta|\Lambda,{\mathbf m}) }$. We sketch the construction. Fix $\eta \in \Omega^s$. For each $v \in V({\mathfrak P})$, let $M_v \cong E(v) \cap E(\kappa)$ index the mobile distinguished marked points on the component labeled $v$ and let $F_v \cong E(v) \backslash E(\kappa)$ index the fixed distinguished marked points on the component labeled $v$. Let $\smash{ \Lambda|_{F_v} } \in \ell^{|F_v|}$ be the collection of points on $\ell$ corresponding under $\ell \cong L_0$ to the points in $\Lambda$ indexed by $\smash{ F_v \hookrightarrow R}$. We agree to write $\smash{ Y^\circ }$ for the open subset of $\smash{ Y \subset |{\sh O}_{{\mathbb P}^2}(\cdot)| }$ where the sections don't vanish identically on $\ell$. Let
$$ V^\circ_\eta|_v =  V_{d_{\mathfrak P}(v),g(v)} (\Lambda|_{F_v}, \mu|_{F_v}; \mu|_{M_v})^\circ. $$
Let $\smash{ V_{\eta}^\circ \subset {\mathbb P}{\sh E}_0 }$ parametrizing $1$-cycles on $W_0$ obtained by gluing appropriately the $1$-cycles in $P$ described above, $|F_v|$ fibers of $F$ with the suitable multiplicities and a curve in $|{\sh O}_F(d_{\mathfrak F}(w)f + \ell)|$ which either intersects $L_0$ at the appropriate points, or completely contains $L_0$, but not $\ell$. Note that the fibers of the map
\begin{equation}\label{ugly map}  
V_{\eta}^\circ \longrightarrow \prod_{v \in V(\Gamma_{\mathfrak P})} V^\circ_\eta|_v 
\end{equation}
are isomorphic to affine lines. Indeed, note that the set of $1$-cycles in $|{\sh O}_F(d_{\mathfrak F}(w)f + \ell)|$ intersecting $L_0$ and $\ell$ in predetermined divisors is an affine space for ${\kk}^\times$ and has a $0$-limit containing $L_0$ (which we allow) and an $\infty$-limit containing $\ell$, which we disallow. Finally, we define $\smash{ \overline{V}_\eta}$ to be the closure of $V_{\eta}^\circ$ inside ${\mathbb P}{\sh E}_0$.

\begin{thm}[Version of {~\cite[Theorem 1.2]{[CH98]}}]\label{CH version} The following hold:

(a) $\smash{ \overline{V}_{d,g,W_0}(\Lambda,{\mathbf m} ) }$ is contained set-theoretically in the union $\smash{ \bigcup_{\eta \in \Omega^s} \overline{V}_\eta }$. 

(b) Moreover, if we choose a general closed point $Y_0 \subset P \cup F$ of some $\overline{V}_\eta$ and, possibly after a base change, a family of curves in $V_{d,g}(\Lambda,{\mathbf m})$ flatly approaching $Y_0$, the central fiber of the family of maps obtained after running semistable reduction can be described as follows: it is the natural gluing of the normalization maps from the components mapping either into $P$ or to the special component in $F$ and maps of degrees $m_i$ from smooth rational curves onto fibers of $F$, totally ramified at the points of intersection with $L_0$ and $\ell$.  
\end{thm}

\begin{proof}[Sketch of reduction.] We sketch part (a) and leave part (b) to the reader. The main point of the reduction is the following completely elementary technical observation whose proof is also left to the reader.

\begin{lem}
Let $V$ be a reduced locally closed subscheme of $|{\sh O}_{{\mathbb P}^2}(d)|$ of pure dimension $N$ such that all elements of $V$ intersect $L$ at the same predetermined (possibly non-reduced) divisor $D \subset L$. Let $\smash{\overline{V}}$ be its closure and $\lim \overline{V}$ the fiber over $0$ of the closure of $V \times {\mathbb A}^\times$ inside the projective bundle ${\mathbb P}{\sh E}$. 

Assume that any \emph{general} element of $\lim \overline{V}$ (which is a nonzero section of ${\sh E}_0$ modulo scalars) has the property that it doesn't vanish identically on $\ell$.

Let $H_p \subset |{\sh O}_{{\mathbb P}^2}(d)|$ be the hyperplane of curves passing through $p$. Note that the hyperplane section $H_p \cap \overline{V}$ is independent of the choice of the point $p \in L \backslash D$ and denote it by $\overline{V}_L \subseteq |{\sh I}_{L/{\mathbb P}^2}(d)| \cong |{\sh O}_{{\mathbb P}^2}(d-1)|$. Then $\lim \overline{V}$ is contained set-theoretically in the projectiviation of the affine cone
$$ \left[\mathrm{Affine\mathop{ }cone\mathop{}over\mathop{}}\overline{V}_L \right] \times_{\mathrm{H}^0({\sh O}_\ell(d-1))} \mathrm{H}^0({\sh I}_{D \times \{0\}/F}(\ell+df) ). $$
\end{lem}

Applying the lemma for $V = V_{d,g}(\Lambda,{\mathbf m})$, the special case of ~\cite[Theorem 1.2]{[CH98]} when all contact points are fixed gives precisely the desired conclusion if we make the assumption that the sections $\sigma \in {\mathrm H}^0(W_0,{\sh E}_0)$ corresponding to general points $[\sigma]$ of $\smash{ \overline{V}_{d,g,W_0}(\Lambda,{\mathbf m}) }$ don't vanish identically on $\ell$. The assumption is indeed true, but it will be justified indirectly. To do that, we argue by descending induction on $g$. Note that the statement is trivial for the maximal $\smash{ g = {d-1 \choose 2} }$. For a fixed $g$, the locus of sections mod scalars of ${\sh E}_0$ in $\smash{ \bigcup_{\eta \in \Omega^s} \overline{V}_\eta }$ which vanish identically on $\ell$ has codimension $2$ in $\smash{ \bigcup_{\eta \in \Omega^s} \overline{V}_\eta }$, because, by connectivity of $G$, there has to be at least one mobile point of intersection with $\ell$. Therefore, the original assumption also holds for $g-1$ for obvious dimension reasons, justifying it for all $g$. 
\end{proof}

\section{Proof of the main theorem}

In this final section, we will give the inductive proof of the main theorem. In \S5.1, after taking one last technical precaution (by restricting to the loci where the fibers of the distinguished evaluation maps are equidimensional, essentially in order to ensure that taking fiber products preserves irreducibility), we use the discussion in \S4 to deduce that any component ${\modu X}$ of the total space of the moduli family contains at least some component of the central fiber from the small landscape. In \S5.2, we use the results of \S2 to prove the key ingredient, namely that whenever ${\modu X}$ contains a component from the small landscape, it also contains the "adjacent" components of the small landscape, in the sense defined in \S3. 

However, there is one more step, which somewhat undermines the word "coalesce" used in the introduction to describe the phenomenon: the components of the degenerate moduli space typically have high multiplicities and our arguments only ensure containment in a topological (multiplicity $1$) sense. Because of this, the right way to argue is by establishing the existence of a component of multiplicity one. Once this is done, our argument ensures that the arbitrary chosen component ${\modu X}$ contains this specific reduced component of the central fiber, so ${\modu X}$ has to be the only component (this final step is similar to the idea used in ~\cite{[Te10]}). The lemma will be treated separately in \S5.3. In \S5.4, we put all the ingredients together and conclude the proof.

From now on, we assume the following:
\begin{equation}\label{inductive assumption}
\text{Assume inductively that Theorem \ref{main theorem} is true for all $d'<d$.}
\end{equation}
In particular, this assumption implies that all $\overline{V}_\eta$ constructed in \S4 are irreducible. 

\subsection{The core of a virtual component}
Assume $\Lambda$ is general. Let $\eta = (\Gamma_{\mathfrak P},\Gamma_{\mathfrak F}) \in \Omega^s$ be a small topological profile. Consider the restricted evaluation maps
\begin{equation} 
{\mathbf q}_{\mathfrak P}: \overline{\modu S}({\mathfrak P}^{\mathrm{rel}},\Gamma_{\mathfrak P}) \to \ell^{|R|} \text{ and } {\mathbf q}_{\mathfrak F}: \overline{\modu S}({\mathfrak F}^{\mathrm{rel}},\Gamma_{\mathfrak F}|\Lambda,{\mathbf m}) \to \ell^{|E(\kappa)|} \hookrightarrow \ell^{|R|},
\end{equation}
where $E(\kappa)$ is the set of roots adjacent to the distinguished $F$-vertex $\kappa \in V({\mathfrak F})$. Note that (\ref{inductive assumption}) also implies that $\smash{ \overline{\modu S}({\mathfrak P}^{\mathrm{rel}},\Gamma_{\mathfrak P}) }$ is irreducible and ${\mathbf q}_{\mathfrak P}$ has generically irreducible geometric fibers, while the analogous statement on the $F$-side is trivial. Let $\smash{ \overline{\modu S}({\mathfrak P}^{\mathrm{rel}} \sqcup {\mathfrak F}^{\mathrm{rel}},\eta|\Lambda, {\mathbf m}) }$ be the image of $\smash{ \overline{\modu S}({\mathfrak P}^\text{rel},\Gamma_{\mathfrak P}) \times_{\ell^{|R|}} \overline{\modu S}({\mathfrak F}^\text{rel},\Gamma_{\mathfrak F}|\Lambda,{\mathbf m}) }$ under the gluing map (\ref{gluing map}). 

By Fact \ref{general fact} in the Appendix, $\smash{ \overline{\modu S}({\mathfrak P}^\mathrm{rel}, \Gamma_{\mathfrak P})|_{\ell^{|E(\kappa)|}}:= ({\mathbf q}_{\mathfrak P})^{-1}(\ell^{|E(\kappa)|}) }$ is irreducible. A similar statement concerning the restriction of $\smash{{\mathbf q}_{\mathfrak P}}$ to $\ell^{|E(\kappa)|}$ holds:
\begin{equation}\label{other evaluation}
{\mathbf q}_{\mathfrak P} : \overline{\modu S}({\mathfrak P}^\mathrm{rel}, \Gamma_{\mathfrak P})|_{\ell^{|E(\kappa)|}} \longrightarrow \ell^{|E(\kappa)|}
\end{equation}
is dominant with generically irreducible fibers. Of course, we have
$$\overline{\modu S}({\mathfrak P}^\mathrm{rel}, \Gamma_{\mathfrak P}) \times_{\ell^{|R|}} \overline{\modu S}({\mathfrak F}^{\mathrm{rel}},\Gamma_{\mathfrak F}|\Lambda,{\mathbf m}) = \overline{\modu S}({\mathfrak P}^\mathrm{rel}, \Gamma_{\mathfrak P})|_{\ell^{|E(\kappa)|}} \times_{\ell^{|E(\kappa)|}} \overline{\modu S}({\mathfrak F}^{\mathrm{rel}},\Gamma_{\mathfrak F}|\Lambda,{\mathbf m}). $$
Keeping in mind Fact \ref{unimportant Li lemma}, let 
$$  {\modu K}^\circ({\mathfrak P}^{\mathrm{rel}} \sqcup {\mathfrak F}^{\mathrm{rel}},\eta|\Lambda,{\mathbf m}) \subseteq  \overline{\modu S}({\mathfrak P}^{\mathrm{rel}} \sqcup {\mathfrak F}^{\mathrm{rel}},\eta|\Lambda, {\mathbf m}) $$ 
be the open substack where the fibers of both ${\mathbf q}_{\mathfrak P}$ and ${\mathbf q}_{\mathfrak F}$ have locally the expected dimension. This open locus will be called the \emph{strict core} of $\smash{ \overline{\modu S}({\mathfrak P}^{\mathrm{rel}} \sqcup {\mathfrak F}^{\mathrm{rel}},\eta|\Lambda,{\mathbf m}) }$. Fact \ref{fiber fact} stated in the Appendix -- or, to be extremely precise, its proof -- implies that it is irreducible. Its closure $\smash{ \overline{\modu K}({\mathfrak P}^{\mathrm{rel}} \sqcup {\mathfrak F}^{\mathrm{rel}},\eta|\Lambda,{\mathbf m}) }$ will be called the main component when we want to contrast it with other potential irreducible components, or simply the core.

\begin{cor}{(of Theorem \ref{CH version})}\label{CH corollary}
Let ${\modu X}$ be any irreducible component of the total space $\smash{ \overline{\modu D}_{g}({\mathfrak W},d|\Lambda,{\mathbf m}) }$ of the degeneration. Then there exists $\eta \in \Omega^s$ such that $\overline{\modu K}({\mathfrak P}^{\mathrm{rel}} \sqcup {\mathfrak F}^{\mathrm{rel}},\eta|\Lambda,{\mathbf m})$ is contained topologically in ${\modu X}$.
\end{cor}

\subsection{Connecting the virtual components} The following proposition is the key ingredient in our argument.

\begin{prop}\label{key trick}
Let $\eta, \eta' \in \Omega^s$ be topological profiles which determine an edge in the graph defined in \S4.2. Let ${\modu X}$ be an irreducible component of $\smash{ \overline{\modu D}_{g}({\mathfrak W},d|\Lambda,{\mathbf m}) }$ containing (in the topological sense) at least one of $\smash{ \overline{\modu K}({\mathfrak P}^{\mathrm{rel}} \sqcup {\mathfrak F}^{\mathrm{rel}},\eta|\Lambda,{\mathbf m}) }$ and $\smash{ \overline{\modu K}({\mathfrak P}^{\mathrm{rel}} \sqcup {\mathfrak F}^{\mathrm{rel}},\eta'|\Lambda,{\mathbf m}) }$. Then ${\modu X}$ contains both of them topologically. \end{prop}

Before proving the proposition, we will use the results in \S2 to establish the following lemma, which operates under the assumptions of Proposition \ref{key trick}. However, the hypotheses of Proposition \ref{key trick} don't play a crucial role in the proof of the lemma and they could have been avoided at the expense of a slightly longer proof.

\begin{lem}\label{two boats}
There exists a map $\smash{ (C,f,x_1,x_2,...,x_n) \in \overline{\modu D}_{g}({\mathfrak W}_0, d|\Lambda,{\mathbf m})(\kk) }$ which belongs to both main components $\smash{ \overline{\modu K}({\mathfrak P}^{\mathrm{rel}} \sqcup {\mathfrak F}^{\mathrm{rel}},\eta|\Lambda,{\mathbf m}) }$ and $\smash{ \overline{\modu K}({\mathfrak P}^{\mathrm{rel}} \sqcup {\mathfrak F}^{\mathrm{rel}},\eta'|\Lambda,{\mathbf m}) }$ and doesn't belong to any other irreducible component of $\smash{ \overline{\modu M}({\mathfrak P}^{\mathrm{rel}} \sqcup {\mathfrak F}^{\mathrm{rel}},\eta|\Lambda,{\mathbf m}) }$ and $\smash{ \overline{\modu M}({\mathfrak P}^{\mathrm{rel}} \sqcup {\mathfrak F}^{\mathrm{rel}},\eta'|\Lambda,{\mathbf m}) }$ respectively.
\end{lem}

\begin{proof}
There are three cases depending on which type of modification the edge $[\eta\eta'] \in E(\Omega^s)$ corresponds to. In each case, we can use the results in \S2 to construct explicitly a degenerate stable map mapping to the first expansion ("accordion") of $W_0$, which belongs to both virtual components. Let $\Lambda = (\xi_1,...,\xi_n) \in L_0^n$.

\tikzset{circ/.style = {draw, circle, minimum size = 5.5mm},}

\begin{tiny}
\begin{center}

\begin{tikzpicture}[scale=0.7]

\def\xd{0}
\def\yd{0}

\draw [thick, draw= black, fill=gray, opacity=0.05] (\xd,\yd) circle [radius=3.2];

\draw [thick, draw=black, fill=gray, opacity=0.1]
       (\xd + 3,\yd -5) to [bend left] (\xd -3, \yd - 5) -- (\xd - 2,\yd -2) -- (\xd + 2,\yd -2) -- cycle;
       
\draw [thick, draw=black, fill=gray, opacity=0.05]
       (\xd + 3,\yd -8) to [bend left] (\xd -3, \yd - 8) -- (\xd - 2,\yd -5) -- (\xd + 2,\yd -5) -- cycle;
\draw [thick, draw=black, opacity = 0.1] (\xd - 2,\yd -8) -- (\xd + 2,\yd -8);

\draw (-3,0) to [in= 180, out = -80] (-1,\yd-2);
\draw (-2.5,0) to [in= 180, out = -70] (-1,\yd-2);

\draw (2,0) to [in=0, out = 80] (-1,\yd-2);
\draw (2,0) to [in=0, out = -100] (-1,\yd-2);

\draw (0,1) to [in= 180, out = -80] (1,\yd-2);
\draw (1.7,1) to [in= 0, out = -100] (1,\yd-2);

\draw [thick,blue] (-1,\yd-2) -- (-1,\yd-5); \draw plot [draw=blue, mark = *, mark size = 2] (-1, \yd - 4);
\draw (1,\yd-2) -- (1,\yd-5);

\draw (-1,\yd-5) to [in = 0, out=180] (-1.8,\yd-8);
\draw (-1,\yd-5) to [in = 180, out=0] (-0.2,\yd-8.3);

\draw (1,\yd-5) to [in = 0, out=180] (-0.2,\yd-8.3);
\draw (1,\yd-5) to [in = 80, out=0] (1.5,\yd-8.3);

\def\c{-7}
\def\yd{-2}

\draw (\c-2,\yd-2.5) to [in= 270, out = 0] (\c-0.5,\yd-2);
\draw (\c-2,\yd-1.5) to [in= 90, out = 0] (\c-0.5,\yd-2);



\draw (\c,\yd-1) to [in= 180, out = -70] (\c+1,\yd-2.5);
\draw (\c+2,\yd-1) to [in= 0, out = -110] (\c+1,\yd-2.5);

\draw [thick,blue] (\c-1,\yd-1) -- (\c-1,\yd-5.5); \draw plot [draw=blue, mark = *, mark size = 2] (\c-1, \yd - 4);

\draw node at (\c-2, \yd - 4) {ram. pt.};

\draw (\c+1,\yd-1) -- (\c+1,\yd-5.5);

\draw (\c-1,\yd-5) to [in = 0, out=180] (\c-1.8,\yd-6);
\draw (\c-1,\yd-5) to [in = 180, out=0] (\c-0.2,\yd-6.3);

\draw (\c+1,\yd-5) to [in = 0, out=180] (\c-0.2,\yd-6.3);
\draw (\c+1,\yd-5) to [in = 180, out=0] (\c+1.5,\yd-6.3);

\draw [->, thick] (-5.5,-6) -- (-2,-4.5);

\def\yd{2}
\def\xd{11}

\def\sa{0.6}

\draw [thick, draw= black, fill=gray, opacity=0.05] (\sa * \xd, \sa* \yd) circle [radius=\sa * 3.2];

\draw [thick, draw=black, fill=gray, opacity=0.05]
       (\sa * \xd + \sa * 3, \sa * \yd - \sa * 5) to [bend left] (\sa * \xd -\sa * 3, \sa * \yd - \sa * 5) -- (\sa * \xd - \sa * 2,\sa * \yd -\sa * 2) -- (\sa * \xd + \sa * 2, \sa * \yd - \sa * 2) -- cycle;
       
\draw (\sa *\xd-\sa *3,\sa *\yd) to [in= 180, out = -80] (\sa *\xd-\sa *1,\sa *\yd-\sa *2);

\draw (\sa *\xd+\sa *2,\sa *\yd) to [in=0, out = 80] (\sa *\xd-\sa *1,\sa *\yd-\sa *2 +1.3);
\draw (\sa *\xd+\sa *2,\sa *\yd) to [in=0, out = -100] (\sa *\xd-\sa *1,\sa *\yd-\sa *2);

\draw (\sa *\xd,\sa *\yd+\sa *1) to [in= 180, out = -80] (\sa *\xd+\sa *1,\sa *\yd-\sa *2);
\draw (\sa *\xd+\sa *1.7,\sa *\yd+\sa *1) to [in= 0, out = -100] (\sa *\xd+\sa *1,\sa *\yd-\sa *2);

\draw (\sa *\xd-\sa *1,\sa *\yd-\sa *2) to [in = 0, out=180] (\sa *\xd-\sa *1.8,\sa *\yd-\sa *5);
\draw (\sa *\xd-\sa *1,\sa *\yd-\sa *2) to [in = 180, out=0] (\sa *\xd-\sa *0.2,\sa *\yd-\sa *5.3);

\draw (\sa *\xd+\sa *1,\sa *\yd-\sa *2) to [in = 0, out=180] (\sa *\xd-\sa *0.2,\sa *\yd-\sa *5.3);
\draw (\sa *\xd+\sa *1,\sa *\yd-\sa *2) to [in = 80, out=0] (\sa *\xd+\sa *1.5,\sa *\yd-\sa *5.3);

\node at (8.7,0) {$\eta'$};

\def\yd{-8}
\def\xd{11}

\def\sa{0.6}

\draw [thick, draw= black, fill=gray, opacity=0.05] (\sa * \xd, \sa* \yd) circle [radius=\sa * 3.2];

\draw [thick, draw=black, fill=gray, opacity=0.05]
       (\sa * \xd + \sa * 3, \sa * \yd - \sa * 5) to [bend left] (\sa * \xd -\sa * 3, \sa * \yd - \sa * 5) -- (\sa * \xd - \sa * 2,\sa * \yd -\sa * 2) -- (\sa * \xd + \sa * 2, \sa * \yd - \sa * 2) -- cycle;
       
\draw (\sa *\xd-\sa *3,\sa *\yd) to [in= 180, out = -80] (\sa *\xd-\sa *1,\sa *\yd-\sa *2);
\draw (\sa *\xd-\sa *2.5,\sa *\yd) to [in= 180, out = -70] (\sa *\xd-\sa *1 + 0.5,\sa *\yd-\sa *2);

\draw (\sa *\xd+\sa *2,\sa *\yd) to [in=0, out = 80] (\sa *\xd-\sa *1 + 0.5,\sa *\yd-\sa *2);
\draw (\sa *\xd+\sa *2,\sa *\yd) to [in=0, out = -100] (\sa *\xd-\sa *1,\sa *\yd-\sa *2);

\draw (\sa *\xd,\sa *\yd+\sa *1) to [in= 180, out = -80] (\sa *\xd+\sa *1,\sa *\yd-\sa *2);
\draw (\sa *\xd+\sa *1.7,\sa *\yd+\sa *1) to [in= 0, out = -100] (\sa *\xd+\sa *1,\sa *\yd-\sa *2);

\draw (\sa *\xd-\sa *1,\sa *\yd-\sa *2) to [in = 0, out=180] (\sa *\xd-\sa *1.8,\sa *\yd-\sa *5);
\draw (\sa *\xd-\sa *1+0.5,\sa *\yd-\sa *2) to [in = 180, out=0] (\sa *\xd-\sa *0.2,\sa *\yd-\sa *5.3);

\draw (\sa *\xd-\sa *1+0.5,\sa *\yd-\sa *2) to [in = 0, out=180]  (\sa *\xd-\sa *1 + 0.25,\sa *\yd-\sa *2-0.5);

\draw (\sa *\xd-\sa *1+0.25,\sa *\yd-\sa *2 -0.5) to [in = 0, out=180]  (\sa *\xd-\sa *1,\sa *\yd-\sa *2);

\draw (\sa *\xd+\sa *1,\sa *\yd-\sa *2) to [in = 0, out=180] (\sa *\xd-\sa *0.2,\sa *\yd-\sa *5.3);
\draw (\sa *\xd+\sa *1,\sa *\yd-\sa *2) to [in = 80, out=0] (\sa *\xd+\sa *1.5,\sa *\yd-\sa *5.3);

\node at (8.7,-6) {$\eta$};

\draw [->, dashed] (4.5,0) -- (3.5,-1);

\draw [->, dashed] (4.5,-6) -- (3.5,-5);

\node at (5.7,1.7) {+1 genus};

\node at (3.6,-2.5) {expansion};

\node at (1.4,-2.3) {$p_{\alpha}$};

\node at (1.4,-4.7) {$p'_{\alpha'}$};

\node at (1.4,-7.7) {$\xi_i$};

\node at (-4.1,-4.8) {$f \cong f'$};

\end{tikzpicture}
\bigskip

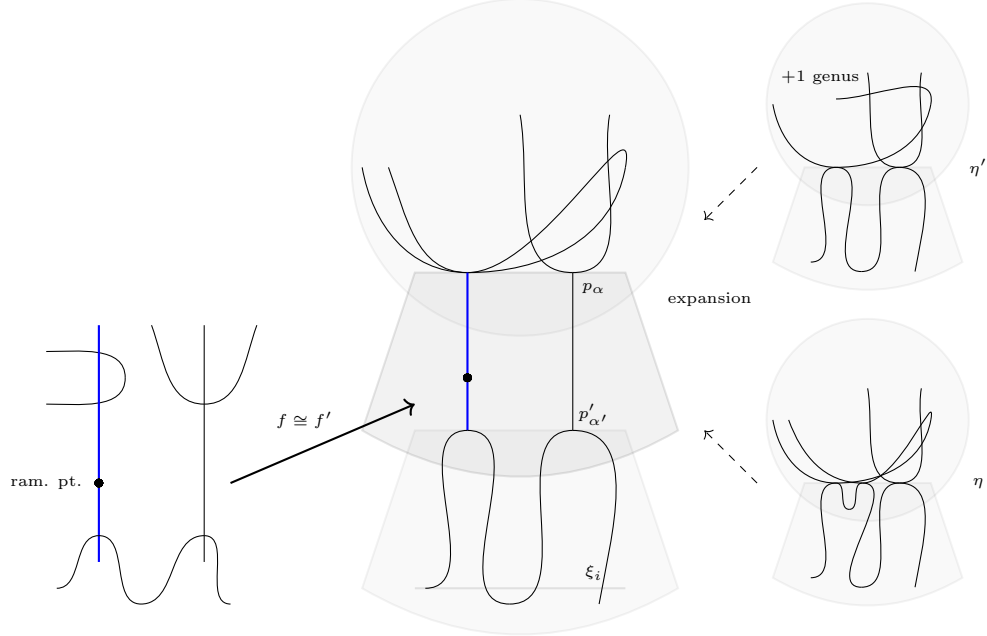
\textbf{Fig. D.} The "transitional" degenerate map for an upper connected operation. \marginpar{\textcolor{white}{ \begin{tiny} improve pic? \end{tiny} } }
\end{center}
\end{tiny}

\emph{Upper connected case.} First, assume that $[\eta\eta'] \in E(\Omega^s)$ corresponds to an upper connected modification. Assume that $\eta'$ is obtained from $\eta$ by identifying two edges connecting the distinguished $F$-vertex $\kappa \in V({\mathfrak F})$ and some vertex $v_0 \in V({\mathfrak P})$. Let $R$ be an indexing set with $|R(\Gamma_{\mathfrak P})| = |R(\Gamma_{\mathfrak F})|$ elements. Let $\tilde{\alpha}_1,\tilde{\alpha}_2 \in R$ be the labels of the two edges which get identified. Let $(p_\alpha)_{\alpha \in R} \in \ell$ be a general collection of closed points on $\ell$ with the restrictions $p_{\tilde{\alpha}_1} = p_{\tilde{\alpha}_2}$ and $p_\alpha \mapsto \xi_i$ under the natural isomorphism $\ell \cong L_0$ if $\nu(i) \in V({\mathfrak P})$ is a degree one vertex whose unique incident edge is precisely the one labeled $\alpha$. Denote by $\smash{ \Pi }$ the tuple of points $(p_{\alpha})_{\alpha \in R} \in \ell^r$. Let $R'$ be a corresponding indexing set for the distinguished marked points (or equivalently, edges of the dual graph) of maps in $\eta'$. The points $\smash{ p'_{\alpha'} }$ satisfy $\smash{ p'_{\rho(\alpha)} = p_\alpha }$ and the collection $\smash{ (p'_{\alpha'})_{\alpha' \in R'} \in \ell^{r'} = \ell^{r-1} }$ is denoted by $\smash{ \Pi' }$. 

Let $\smash{ (f_{\mathfrak P},f'_{\mathfrak P}) }$ be the pair of convergent doppelg\"{a}ngers constructed in Proposition \ref{UC dopelganger} and $\smash{ (f_{\mathfrak F},f'_{\mathfrak F}) }$ the pair of divergent doppelg\"{a}ngers constructed in Lemma \ref{F-dopelganger divergent}. \marginpar{\textcolor{white}{ \begin{tiny} Explain more! \end{tiny} } }On one hand, the relative stable maps $\smash{ f_{\mathfrak P} }$ of type $\smash{ \Gamma_{\mathfrak P} }$ and $\smash{ f_{\mathfrak F} }$ of type $\smash{ \Gamma_{\mathfrak F} }$ glue to a degenerate stable map 
$$ f \in \overline{\modu M}({\mathfrak P}^{\mathrm{rel}} \sqcup {\mathfrak F}^{\mathrm{rel}}, \eta |\Lambda,{\mathbf m})(\kk) = \overline{\modu M}({\mathfrak W}_0, \eta |\Lambda,{\mathbf m})(\kk). $$
On the other hand, the relative stable maps $\smash{ f'_{\mathfrak P} }$ of type $\smash{ \Gamma'_{\mathfrak P} }$ and $\smash{ f'_{\mathfrak F} }$ of type $\smash{ \Gamma'_{\mathfrak F} }$ also glue to a degenerate stable map 
$$ f' \in \overline{\modu M}({\mathfrak P}^{\mathrm{rel}} \sqcup {\mathfrak F}^{\mathrm{rel}}, \eta' |\Lambda,{\mathbf m})(\kk) = \overline{\modu M}({\mathfrak W}_0, \eta' |\Lambda,{\mathbf m})(\kk). $$
However (keeping in mind Remark \ref{branch point}), $f$ and $f'$ are isomorphic by construction, so we can regard them as the same map. Thanks to the good behavior of all $4$ pieces, we conclude that $\smash{ f \cong f' }$ belongs to both strict cores $\smash{ {\modu K}^\circ({\mathfrak P}^{\mathrm{rel}} \sqcup {\mathfrak F}^{\mathrm{rel}},\eta|\Lambda,{\mathbf m}) }$ and $\smash{ {\modu K}^\circ({\mathfrak P}^{\mathrm{rel}} \sqcup {\mathfrak F}^{\mathrm{rel}},\eta'|\Lambda,{\mathbf m}) }$, as desired. 

By the assumption in Proposition \ref{key trick}, at least one of $\smash{ \overline{\modu K}({\mathfrak P}^{\mathrm{rel}} \sqcup {\mathfrak F}^{\mathrm{rel}},\eta|\Lambda,{\mathbf m}) }$ and $\smash{ \overline{\modu K}({\mathfrak P}^{\mathrm{rel}} \sqcup {\mathfrak F}^{\mathrm{rel}},\eta'|\Lambda,{\mathbf m}) }$ is contained inside $\smash{ \overline{\modu D}_{g}({\mathfrak W}_0, d|\Lambda,{\mathbf m}) }$,\footnote{Of course, in reality they are both contained in $\smash{ \overline{\modu D}_{g}({\mathfrak W}_0, d|\Lambda,{\mathbf m}) }$, but this is the part of the argument we've chosen to simplify relying on the (logically justified) assumptions in Proposition \ref{key trick}.} hence $f \cong f'$ indeed belongs to $\smash{ \overline{\modu D}_{g}({\mathfrak W}_0, d|\Lambda,{\mathbf m}) }$, i.e. the degenerate stable map we've constructed is a limit of stable maps corresponding to nice curves (Definition \ref{nice curve}) in the isotrivial family of generalized Severi varieties.

\begin{center}
\begin{tikzpicture}
\matrix [column sep  = 8mm, row sep = 8mm] {
	\node (nnw) {$\overline{\modu S}({\mathfrak P}^\mathrm{rel}, \Gamma'_{\mathfrak P})$}; &

	\node (nw) {$\overline{\modu M}({\mathfrak P}^\mathrm{rel}, \Gamma'_{\mathfrak P})$}; &
	\node (ne) {$\overline{\modu M}({\mathfrak F}^{\mathrm{rel}},\Gamma'_{\mathfrak F}|\Lambda,{\mathbf m})$}; &
	\node (nne) {$\overline{\modu S}({\mathfrak F}^{\mathrm{rel}},\Gamma'_{\mathfrak F}|\Lambda,{\mathbf m})$}; \\ &
	\node (sw) {$\ell^{|R'|}$}; &
	\node (se) {$\ell^{|E(\kappa')|}$}; \\
};
\draw[right hook->, thin] (nnw) -- (nw);
\draw[left hook->, thin] (nne) -- (ne);
\draw[left hook->, thin] (se) -- (sw);
\draw[->, thin] (nw) -- (sw);
\draw[->, thin] (ne) -- (sw);
\draw[->, thin] (nne) -- (se);

\node (ch) at (0,3.4) {$\overline{\modu S}({\mathfrak P}^\mathrm{rel}, \Gamma'_{\mathfrak P}) \times_{\ell^{|R'|}} \overline{\modu S}({\mathfrak F}^{\mathrm{rel}},\Gamma'_{\mathfrak F}|\Lambda,{\mathbf m})$};

\node (cl) at (0,2) {$\overline{\modu M}({\mathfrak P}^\mathrm{rel}, \Gamma'_{\mathfrak P}) \times_{\ell^{|R'|}} \overline{\modu M}({\mathfrak F}^{\mathrm{rel}},\Gamma'_{\mathfrak F}|\Lambda,{\mathbf m})$};

\draw[->, thin] (cl) -- (nw);
\draw[->, thin] (cl) -- (ne);
\draw[->, thin] (ch) to[out=180,in=90]  (nnw);
\draw[->, thin] (ch) to[out=0,in=90] (nne);
\draw[left hook->, thin] (ch) -- (cl);

\node at (-2.6,0) {${\mathbf q}'_{\mathfrak P}$};
\node at (-1.1,0) {${\mathbf q}'_{\mathfrak F}$};

\end{tikzpicture}
\end{center} 
We need to check that the sole irreducible component of $\smash{ \overline{\modu M}({\mathfrak P}^{\mathrm{rel}} \sqcup {\mathfrak F}^{\mathrm{rel}},\eta'|\Lambda,{\mathbf m}) }$ containing $\smash{ (C',f',x'_1,x'_2,...,x'_n) }$ is the core $\smash{ \overline{\modu K}({\mathfrak P}^{\mathrm{rel}} \sqcup {\mathfrak F}^{\mathrm{rel}},\eta'|\Lambda,{\mathbf m}) }$. That is, we need to check that the immersion $\smash{ \overline{\modu K}({\mathfrak P}^{\mathrm{rel}} \sqcup {\mathfrak F}^{\mathrm{rel}},\eta'|\Lambda,{\mathbf m}) \hookrightarrow \overline{\modu M}({\mathfrak P}^{\mathrm{rel}} \sqcup {\mathfrak F}^{\mathrm{rel}},\eta'|\Lambda,{\mathbf m}) }$ is not only closed, but also open near $[f']$. By Fact \ref{unimportant Li lemma}, it suffices to prove that $([f'_{\mathfrak P}],[f'_{\mathfrak F}])$ only belongs to the main irreducible component of 
$$ \overline{\modu M}({\mathfrak P}^\mathrm{rel}, \Gamma'_{\mathfrak P}) \times_{\ell^{|R'|}} \overline{\modu M}({\mathfrak F}^{\mathrm{rel}},\Gamma'_{\mathfrak F}|\Lambda,{\mathbf m}). $$
This is the unique component contained in $\smash{ \overline{\modu S}({\mathfrak P}^\mathrm{rel}, \Gamma'_{\mathfrak P}) \times_{\ell^{|R'|}} \overline{\modu S}({\mathfrak F}^{\mathrm{rel}},\Gamma'_{\mathfrak F}|\Lambda,{\mathbf m}) }$ which dominates $\smash{ \ell^{|E(\kappa')|} \subset \ell^{|R'|} }$.  Recall that $E(\kappa')$ is the set of edges (or equivalently, roots), incident to the distinguished $F$-vertex $\kappa'$. 

At the very least, it is clear that any irreducible component containing $\smash{ ([f'_{\mathfrak P}],[f'_{\mathfrak F}]) }$ must be contained in $\smash{ \overline{\modu S}({\mathfrak P}^\mathrm{rel}, \Gamma'_{\mathfrak P}) \times_{\ell^{|R'|}} \overline{\modu S}({\mathfrak F}^{\mathrm{rel}},\Gamma'_{\mathfrak F}|\Lambda,{\mathbf m}) }$ because $\smash{ [f'_{\mathfrak P}] }$ only belongs to $\smash{ \overline{\modu S}({\mathfrak P}^\mathrm{rel}, \Gamma'_{\mathfrak P}) }$ among all irreducible components of $\smash{ \overline{\modu M}({\mathfrak P}^\mathrm{rel}, \Gamma'_{\mathfrak P}) }$ by Proposition \ref{UC dopelganger} and similarly, $\smash{ [f'_{\mathfrak F}] }$ only belongs to $\smash{ \overline{\modu S}({\mathfrak F}^\mathrm{rel}, \Gamma'_{\mathfrak F}|\Lambda,{\mathbf m}) }$ among all irreducible components of $\smash{ \overline{\modu M}({\mathfrak F}^\mathrm{rel}, \Gamma'_{\mathfrak F}|\Lambda,{\mathbf m}) }$ by Lemma \ref{F-dopelganger divergent}. 

To review, we know that the two factors in the fibered product above are irreducible mapping dominantly onto the base with generically irreducible geometric fibers and that the fibers have locally the expected dimension at both $\smash{ [f'_{\mathfrak P}] }$ and $\smash{ [f'_{\mathfrak F}] }$, hence $\smash{ ([f'_{\mathfrak P}],[f'_{\mathfrak F}]) }$ necessarily lives in the main component of this fiber product and \emph{only} in it and the conclusion follows. The argument is outlined in Fact \ref{fiber fact} in the Appendix.
 
\begin{center}
\begin{tikzpicture}
\matrix [column sep  = 8mm, row sep = 8mm] {
	\node (nnw) {$\overline{\modu S}({\mathfrak P}^\mathrm{rel}, \Gamma_{\mathfrak P})$}; &

	\node (nw) {$\overline{\modu M}({\mathfrak P}^\mathrm{rel}, \Gamma_{\mathfrak P})$}; &
	\node (ne) {$\overline{\modu M}({\mathfrak F}^{\mathrm{rel}},\Gamma_{\mathfrak F}|\Lambda,{\mathbf m})$}; &
	\node (nne) {$\overline{\modu S}({\mathfrak F}^{\mathrm{rel}},\Gamma_{\mathfrak F}|\Lambda,{\mathbf m})$}; \\
	\node (ssw) {$\ell^{|R'|} = \ell^{|R|-1}$}; &
	\node (sw) {$\ell^{|R|}$}; &
	\node (se) {$\ell^{|E(\kappa)|}$}; \\
};
\draw[right hook->, thin] (nnw) -- (nw);
\draw[right hook->, thin] (ssw) -- (sw);
\draw[left hook->, thin] (nne) -- (ne);
\draw[left hook->, thin] (se) -- (sw);
\draw[->, thin] (nw) -- (sw);
\draw[->, thin] (ne) -- (sw);
\draw[->, thin] (nne) -- (se);

\node at (-3.2,-0.5) {$\Delta_{\tilde{\alpha}_1,\tilde{\alpha}_2}$};

\node (ch) at (0,3.4) {$\overline{\modu S}({\mathfrak P}^\mathrm{rel}, \Gamma_{\mathfrak P}) \times_{\ell^{|R|}} \overline{\modu S}({\mathfrak F}^{\mathrm{rel}},\Gamma_{\mathfrak F}|\Lambda,{\mathbf m})$};

\node (cl) at (0,2) {$\overline{\modu M}({\mathfrak P}^\mathrm{rel}, \Gamma_{\mathfrak P}) \times_{\ell^{|R|}} \overline{\modu M}({\mathfrak F}^{\mathrm{rel}},\Gamma_{\mathfrak F}|\Lambda,{\mathbf m})$};

\draw[->, thin] (cl) -- (nw);
\draw[->, thin] (cl) -- (ne);
\draw[->, thin] (ch) to[out=180,in=90]  (nnw);
\draw[->, thin] (ch) to[out=0,in=90] (nne);
\draw[left hook->, thin] (ch) -- (cl);

\node at (-2.6,0) {${\mathbf q}_{\mathfrak P}$};
\node at (-1.1,0) {${\mathbf q}_{\mathfrak F}$};

\end{tikzpicture}
\end{center} 

We need to check the analogous statement with $\eta$ in place of $\eta'$, i.e. that the only irreducible component of $\smash{ \overline{\modu M}({\mathfrak P}^{\mathrm{rel}} \sqcup {\mathfrak F}^{\mathrm{rel}},\eta|\Lambda,{\mathbf m}) }$ containing $[f]$ is the main component $\smash{ \overline{\modu K}({\mathfrak P}^{\mathrm{rel}} \sqcup {\mathfrak F}^{\mathrm{rel}},\eta|\Lambda,{\mathbf m}) }$. This is ultimately very similar to the case of $\eta'$, with the additional feature that the fiber of $\smash{ {\mathbf q}_\eta }$ containing $[f]$ needs to satisfy the non-generic condition $\smash{ p_{\tilde{\alpha}_1} = p_{\tilde{\alpha}_2} }$. However, all the relevant fiber products are still over $\smash{ \ell^{|R|} }$ rather than the diagonal and it turns out that this complication is completely irrelevant. The argument in the previous case goes through essentially word for word after erasing the "$'$s". 

\emph{Upper disconnected case.} This is essentially identical to the previous case.

\emph{Lower disconnected case.} This is similar to the upper connected case as well. Instead of \ref{UC dopelganger} and \ref{F-dopelganger divergent}, we will need to use \ref{plane divergent doppelganger} and \ref{F-dopelganger convergent}. This is left to the reader.
\end{proof}

\begin{sublem}[{\cite[Lemma 3.4]{[Li02]}}]\label{key cite} There exists a pair $({\mathbf L}_\eta,s_\eta)$ consisting of a line bundle on $\smash{ \overline{\modu M}_{g,n}({\mathfrak W},d) }$ and a section vanishing topologically precisely on the virtual component $\smash{ \overline{\modu M}({\mathfrak P}^\mathrm{rel} \sqcup {\mathfrak F}^\mathrm{rel},\eta) }$.

\end{sublem}

\begin{proof}[Proof of Proposition \ref{key trick}.] To clarify notation, $\smash{ \eta }$ and $\smash{ \eta' }$ are completely interchangeable -- we're not imposing the asymmetry inherent in \S3.2. Assume that $\smash{\modu X}$ contains $\smash{ \overline{\modu K}({\mathfrak P}^\mathrm{rel} \sqcup {\mathfrak F}^\mathrm{rel},\eta'|\Lambda,{\mathbf m}) }$ topologically. 

Let $\smash{ ({\mathbf L}_\eta,s_\eta) }$ be the pair consisting of a line bundle on $\smash{ \overline{\modu M}_{g,n}({\mathfrak W},d) }$ and a section vanishing topologically precisely on $\smash{ \overline{\modu M}({\mathfrak W}_0,\eta) }$. Let $\smash{ {\mathbf L}_\eta^{\modu X} }$ and $\smash{ s_\eta^{\modu X} }$ be the restrictions to $\smash{\modu X}$ of $\smash{ {\mathbf L}_\eta }$ and $\smash{ s_\eta }$. By definition, the section $\smash{ s_\eta^{\modu X} }$ vanishes at the point $\smash{ [f] }$ corresponding to the degenerate stable map constructed in Lemma \ref{two boats}. Locally, it must vanish along a substack of dimension at least $\smash{ \dim {\modu X} - 1 = 2d+g-1 }$. In other words, the intersection of $\smash{\modu X}$ with $\smash{ \overline{\modu M}({\mathfrak W}_0,\eta|\Lambda,{\mathbf m}) }$ has dimension at least $2d+g-1$ locally near $\smash{ [f] }$. However, since $\smash{ [f] }$ belongs to no other irreducible component of $\smash{ \overline{\modu M}({\mathfrak P}^\mathrm{rel} \sqcup {\mathfrak F}^\mathrm{rel},\eta|\Lambda,{\mathbf m}) }$ except $\smash{ \overline{\modu K}({\mathfrak P}^\mathrm{rel} \sqcup {\mathfrak F}^\mathrm{rel},\eta|\Lambda,{\mathbf m}) }$, which is irreducible of dimension $2d+g-1$, we conclude that ${\modu X}$ must contain $\smash{ \overline{\modu K}({\mathfrak P}^\mathrm{rel} \sqcup {\mathfrak F}^\mathrm{rel},\eta|\Lambda,{\mathbf m}) }$ as well (in the topological sense), completing the proof. \end{proof}

\subsection{Smooth points of the central fiber} In this subsection, we will provide the final ingredient by showing that at least some points in the central fiber of the moduli family are smooth in the total space. Very roughly, the general philosophy is that the multiplicity of a certain component should equal the product of the weights at the distinguished marked points and, translated into our language, ~\cite[Theorem 1.3]{[CH98]} says precisely that. However, because the translation process is awkward and because we don't need the full strength of the Caporaso--Harris theorem, we will instead give an improvised direct proof of the weaker statement we require.

\begin{prop}\label{multiplicity one}
There exists a profile $\eta_0 \in \Omega^s$ such that $\smash{ \overline{\modu M}_{g}({\mathfrak W},d|\Lambda,{\mathbf m}) }$ is smooth at some point $\smash{ [f] \in \overline{\modu K}({\mathfrak P}^\mathrm{rel} \sqcup {\mathfrak F}^\mathrm{rel},\eta_0|\Lambda,{\mathbf m})(\kk) }$.
\end{prop}

\begin{proof}
We claim that any small profile $\eta_0 \in \Omega^s$ such that $\mu \equiv 1$ and $|V({\mathfrak F})| = 1$ satisfies the desired property. It is straightforward to check that such profiles exist. A general element $f$ of $\smash{ \overline{\modu K}({\mathfrak P}^\mathrm{rel} \sqcup {\mathfrak F}^\mathrm{rel},\eta_0|\Lambda,{\mathbf m}) }$ has source $\smash{ C = C_\kappa \cup \bigcup_{v \in V({\mathfrak P})} C_v }$ glued appropriately, with $\smash{ C_\kappa }$ mapping into $F$ and the $\smash{ C_v }$ mapping into $P$. The restriction of $f$ to $\smash{ C_v }$ is simply the composition of the normalization map with the inclusion into $P$, while the restriction of $f$ to $\smash{ C_\kappa }$ is just a closed embedding. We will sometimes abuse notation for normal and conormal sheaves, specifying the curves but not the maps. Note that $f$ is an unramified map, so describing the first order deformations of $f$ should be fairly straightforward. Our improvised approach is simply to study the first order deformations of $f$ as a (standard) stable map into $W$. Of course, without the assumptions that $\mu \equiv 1$ and that $f$ is sufficiently nice, this would be completely wrong. It is a straightforward technical exercise to verify that the moduli spaces
$$ \overline{\modu M}_{g}({\mathfrak W},d|\Lambda,{\mathbf m}) \text{ and } \overline{\modu M}_{g}(W,d|\Lambda,{\mathbf m}) $$
are isomorphic \'{e}tale locally near $f$, where $\smash{ \overline{\modu M}_{g}(W,d|\Lambda,{\mathbf m}) }$ is the closed substack of ${ \overline{\modu M}_{g,n}(W,d) }$ parametrizing maps which send the $i$th marked point to $p_i$ and such that the pullback of $1 \in \mathrm{H}^0({\sh O}_{W}(L \times {\mathbb A}^1))$ to the source of the map vanishes to order at least $m_i$ at the $i$th marked point.

We will regard the divisor $D = \sum m_ip_i$ as a closed subscheme of the source curve of $f$. The vector space of first order deformations of $f$ as an unmarked stable map into $W$, i.e. as a $\kk$-point of $\smash{ \overline{\modu M}_{g,0}(W,d) }$, is naturally the space of global sections of the lci normal sheaf ${\sh N}_f$. The first order deformations of $f$ in $\smash{ \overline{\modu M}_{g}(W,d|\Lambda,{\mathbf m}) }$ are naturally identified with $\mathrm{H}^0({\sh F})$, where ${\sh F} \subset {\sh N}_f$ is the kernel of the ${\sh O}_C$-module homomorphism
\begin{equation}\label{(5.1)}
{\sh N}_f \longrightarrow \iota_* {\sh O}_D
\end{equation}
obtained by contracting an arbitrary section $\sigma \in {\sh N}_f(U)$ with the pullback of the differential form $\mathrm{d}t \in \Omega^1\mathrm{Spec}({\kk}[t])$. Note that the contraction is well defined because the tangent bundle of $C$ is mapped by (the derivative of) $f$ into the central fiber, where $\mathrm{d}t$ vanishes identically.

We have to compute $h^0({\sh F})$. Fix an arbitrary $v \in V({\mathfrak P})$. Note that the restriction of $\smash{ {\sh N}^\vee_f }$ and thus also of $\smash{ {\sh F}^\vee }$ to $\smash{ C_v }$ is isomorphic to the elementary modification of the rank $2$ bundle $\smash{ {\sh N}^\vee_{C_v/W} }$ at the distinguished marked points living on $\smash{C_v}$, in the direction of the (cosets of the) conormal lines of $\smash{C_\kappa}$ at the respective points, i.e. the kernel of the ${\sh O}_{C_v}$-module homomorphism
\begin{equation}\label{elem.mod.1} 
{\sh N}^\vee_{C_v/W} \longrightarrow \iota_* {\sh O}_{C_v \cap \ell}
\end{equation}
obtained by contracting with (arbitrary nonzero) vectors tangent to $C_\kappa$ at all points $C_v \cap \ell$. Note that the map (\ref{elem.mod.1}) restricts to a surjective ${\sh O}_{C_v}$-module homomorphism ${\sh N}^\vee_{P/W}|_{C_v} \to \iota_* {\sh O}_{C_v \cap \ell}$ whose kernel is $\smash{ {\sh N}_{P/W}^\vee|_{C_v} \otimes {\sh I}_{C_v \cap \ell/C_v} }$, which is isomorphic to the structure sheaf $\smash{ {\sh O}_{C_v} }$ by adjunction. Applying the snake lemma to the diagram
\begin{center}
\begin{tikzpicture}
\matrix [column sep  = 8mm, row sep = 6mm] {

	\node (nww) {$0$}; &
	\node (nw) {${\sh N}_{P/W}^\vee|_{C_v}$}; &
	\node (nc) {${\sh N}^\vee_{C_v/W}$}; &
	\node (ne) {${\sh N}^\vee_{C_v/P}$};  &
	\node (nee) {$0$}; \\
	\node (sww) {$0$}; &
	\node (sw) {$\displaystyle{ \iota_* {\sh O}_{C_v \cap \ell} }$}; &
	\node (sc) {$\displaystyle{ \iota_* {\sh O}_{C_v \cap \ell} }$}; &
	\node (se) {$\displaystyle{ 0 }$}; \\
};

\draw[->, thin] (nww) -- (nw);
\draw[->, thin] (nw) -- (nc);
\draw[->, thin] (nc) -- (ne);
\draw[->, thin] (ne) -- (nee);

\draw[->, thin] (sww) -- (sw);
\draw[double equal sign distance, thin] (sw) -- (sc);
\draw[->, thin] (sc) -- (se);

\draw[->, thin] (nw) -- (sw);
\draw[->, thin] (nc) -- (sc);
\draw[->, thin] (ne) -- (se);

\node at (0.6,0) {(\ref{elem.mod.1})};

\end{tikzpicture}
\end{center} 
in which the top row is the conormal sheaf sequence for $C_v \to P \subset W$, we obtain the sequence 
\begin{equation}
\label{(5.2)} 0 \longrightarrow {\sh O}_{C_v} \longrightarrow {\sh F}^\vee|_{C_v} \longrightarrow {\sh N}^\vee_{C_v/P} \longrightarrow 0. 
\end{equation}
From the short exact sequence for the normal bundle of $C_v \to P$, we compute ${\sh N}^\vee_{C_v/P} \cong \omega^\vee_{C_v}(-3)$. Note that 
$$ \mathrm{Ext}^1({\sh T}_{C_v}(-3), {\sh O}_{C_v}) = \mathrm{H}^1(\omega_{C_v}(3)) = 0 $$ 
by Serre duality, so (\ref{(5.2)}) splits. In conclusion, the restriction of $\smash{ {\sh F} }$ to $C_v$ is isomorphic to ${\sh O}_{C_v} \oplus \omega_{C_v}(3)$. Moreover, the fiber of the maximal destabilizing subsheaf $\omega_{C_v}(3)$ at any $p \in C_v \cap \ell$ can be described as the kernel of the map ${\sh F} \otimes {\kk}_p = {\sh N}_f \otimes {\kk}_p \to {\sh N}_{P \cup F/W} \otimes {\kk}_p$.

Similarly, the restriction of $\smash{ {\sh N}_f^\vee }$ to $C_\kappa$ is isomorphic to the kernel of the homomorphism
\begin{equation}\label{elem.mod.2} 
{\sh N}^\vee_{C_\kappa/W} \longrightarrow \iota_* {\sh O}_{C_\kappa \cap \ell}
\end{equation}
obtained by contracting with vectors tangent to $C_v$ at all points $C_\kappa \cap \ell$, over all $v \in V({\mathfrak P})$. Essentially by repeating the argument in the paragraph above, we obtain a short exact sequence analogous to (\ref{(5.2)})
\begin{equation}\label{(5.3)} 
0 \longrightarrow {\sh O}_{C_\kappa} \longrightarrow {\sh N}_f^\vee|_{C_\kappa} \longrightarrow {\sh N}^\vee_{C_\kappa/F} \longrightarrow 0. 
\end{equation}
and by yet another extension group computation based on Kodaira--Serre duality, once again we see that the short exact sequence splits. Moreover, ${\sh N}_{C_\kappa/F} \cong \omega_{C_\kappa} \otimes f_{\kappa}^* \omega_F^\vee$, so in conclusion ${\sh N}_f|_{C_\kappa} \cong {\sh O}_{C_\kappa} \oplus \omega_{C_\kappa} \otimes f_{\kappa}^* \omega_F^\vee$. However, on this occasion, we also need to factor in the contact with $L_0$ along $D$. To do this, we resort once more to the pattern involving the snake lemma:
\begin{center}
\begin{tikzpicture}
\matrix [column sep  = 8mm, row sep = 6mm] {

	\node (nww) {$0$}; &
	\node (nw) {${\sh N}_{C_\kappa/F}$}; &
	\node (nc) {${\sh N}_f|_{C_\kappa}$}; &
	\node (ne) {${\sh O}_{C_\kappa}$};  &
	\node (nee) {$0$}; \\
	\node (sww) {$0$}; &
	\node (sw) {$\displaystyle{ \iota_* {\sh O}_D }$}; &
	\node (sc) {$\displaystyle{ \iota_* {\sh O}_D }$}; &
	\node (se) {$\displaystyle{ 0 }$}; \\
};

\draw[->, thin] (nww) -- (nw);
\draw[->, thin] (nw) -- (nc);
\draw[->, thin] (nc) -- (ne);
\draw[->, thin] (ne) -- (nee);

\draw[->, thin] (sww) -- (sw);
\draw[double equal sign distance, thin] (sw) -- (sc);
\draw[->, thin] (sc) -- (se);

\draw[->, thin] (nw) -- (sw);
\draw[->, thin] (nc) -- (sc);
\draw[->, thin] (ne) -- (se);

\node at (0.6,0) {(\ref{(5.1)})};

\end{tikzpicture}
\end{center} 
in which the top row is obviously the dual of (\ref{(5.3)}). We obtain a short exact sequence
\begin{equation*}
0 \longrightarrow {\sh N}_{C_\kappa/F}(-D) \longrightarrow {\sh F}|_{C_\kappa} \longrightarrow {\sh O}_{C_\kappa}  \longrightarrow 0. 
\end{equation*}
By Serre duality again, $\mathrm{Ext}^1({\sh O}_{C_\kappa},\omega_{C_\kappa}(-D) \otimes f_{\kappa}^* \omega_F^\vee) = \mathrm{H}^0(f_\kappa^*\omega_F(D)) = 0$ since $f_\kappa^*\omega_F(D)$ has strictly negative degree $-d-1$. In conclusion, the restriction of ${\sh F}$ to the curve $C_{\kappa}$ is isomorphic to the direct sum ${\sh O}_{C_\kappa} \oplus \omega_{C_\kappa}(-D) \otimes f_{\kappa}^* \omega_F^\vee$. Again, the fiber of the maximal destabilizing subsheaf $\omega_{C_\kappa}(-D) \otimes f_{\kappa}^* \omega_F^\vee$ at any $p \in C_\kappa \cap \ell$ is the kernel of the map ${\sh F} \otimes {\kk}_p \cong {\sh N}_f \otimes {\kk}_p \to {\sh N}_{P \cup F/W} \otimes {\kk}_p$. At this point we're almost done, so we isolate the essential final ingredient.

\begin{sublem}\label{final fact}
Let ${\sh F}$ be a rank $2$ locally free sheaf on a semistable reducible curve $C$ whose irreducible components are all smooth. Assume that for any irreducible component $Y$ of $C$,
$$ {\sh F}|_Y \cong {\sh O_Y} \oplus {\sh L}_Y, $$
where $ {\sh L}_Y \in \mathrm{Pic}(Y)$ is a line bundle of degree
$$ \deg {\sh L}_Y > \deg \omega_Y + |Y \cap (C \backslash Y)|. $$
Furthermore, assume that for any node $p = Y_1 \cap Y_2$ of $C$, the $p$-fiber of the maximal destabilizing subsheaf ${\sh L}_{Y_1}$ of ${\sh F}|_{Y_1}$ coincides with the $p$-fiber of the maximal destabilizing subsheaf ${\sh L}_{Y_2}$ of ${\sh F}|_{Y_2}$. Then
\begin{equation}\label{h0 form}
h^0({\sh F}) = 1 + \sum_{Y} h^0({\sh L}_Y) - \text{number of nodes of }C.
\end{equation}
\end{sublem}

\begin{proof}
The proof is straightforward. The second condition trivially implies that the maximal destabilizing subsheaves glue to a line bundle ${\sh L}$ and then a formal argument similar to the first half of the proof of Proposition \ref{UC dopelganger} shows that $h^0({\sh L})$ is indeed one less the right hand side of (\ref{h0 form}), as desired.
\end{proof}

Using Fact \ref{final fact}, we conclude that
\begin{equation*}
\begin{aligned}
h^0({\sh F}) &= 1 + \sum_{v \in V({\mathfrak P})} h^0(\omega_{C_v}(3)) + h^0( \omega_{C_\kappa}(-D) \otimes f_\kappa^*\omega^\vee_F ) - (d-1) \\
&= 1 + \sum_{v \in V({\mathfrak P})} (1 + 3d_{\mathfrak P}(v) + 2g(v)-2 - g(v) ) + d - (d-1) \\
&= 1+ 3(d-1) + \sum_{v \in V({\mathfrak P})} (g(v) - 1) + 1 = 2d+g \\
\end{aligned}
\end{equation*}
by (\ref{topological requirement}), so the space of first order deformations of $f$ as a map in $\overline{\modu M}_{g}(W,d|\Lambda,{\mathbf m})$ has dimension $2d+g$. By the discussion in the beginning of the proof, the same will hold for the space of first order deformations of $f$ as a map in $\overline{\modu M}_{g}({\mathfrak W},d|\Lambda,{\mathbf m})$, which is the expected dimension. Hence $\smash{\overline{\modu M}_{g}({\mathfrak W},d|\Lambda,{\mathbf m}) }$ has the expected dimension at $f$, as desired.
\end{proof}

Note that $h^1({\sh F}) \neq 0$, so $f$ is actually obstructed as an element of $\overline{\modu M}_{g}(W,d|\Lambda,{\mathbf m})$. However, this is merely means that the obstruction space is "incorrect" which is a symptom of the fact that the improvised method we used is quite artificial.

\subsection{The inductive procedure} In this final subsection, we put all the ingredients together and conclude the proof of the Theorem \ref{main theorem}.

\begin{proof}[Proof of Theorem \ref{main theorem}] Recall the ongoing inductive assumption (\ref{inductive assumption}) that \ref{main theorem} is true for all $d'<d$. By the observation after Lemma \ref{early def th}, it suffices to prove the case when all contact points are fixed, i.e. we need to prove that $\smash{ \overline{V}_{d,g}(\Lambda, {\mathbf m}) }$ is irreducible. Consider again $\overline{\modu D}_{g}({\mathfrak W},d|\Lambda,{\mathbf m}) \to {\mathbb A}^1$, the total space of the degeneration of the moduli space. Recall that all the cores $\smash{ \overline{\modu K}({\mathfrak P}^\mathrm{rel} \sqcup {\mathfrak F}^\mathrm{rel},\eta|\Lambda,{\mathbf m}) }$ of the virtual components in the small landscape are irreducible of dimension $2d+g-1$.

Let ${\modu X}$ be an irreducible component of $\smash{ \overline{\modu D}_{g}({\mathfrak W},d|\Lambda,{\mathbf m}) }$. Let $\eta_0$ be the small topological profile satisfying the condition in Proposition \ref{multiplicity one}. We claim that ${\modu X}$ contains $\smash{ \overline{\modu K}({\mathfrak P}^\mathrm{rel} \sqcup {\mathfrak F}^\mathrm{rel},\eta_0|\Lambda,{\mathbf m}) }$. By Corollary \ref{CH corollary} and the statement in the paragraph above, there exists a small topological profile $\eta_1$ such that $\smash{ \overline{\modu K}({\mathfrak P}^\mathrm{rel} \sqcup {\mathfrak F}^\mathrm{rel},\eta_1|\Lambda,{\mathbf m}) }$ is contained topologically in ${\modu X}$. Proposition \ref{landscape is connected} ensures that there exists a sequence $\eta_1,\eta_2,...,\eta_M=\eta_0$ of small topological profiles such that $[\eta_i\eta_{i+1}] \in E(\Omega^s)$, for all $i<M$. Inductively, by Proposition \ref{key trick}, ${\modu X}$ contains $\smash{ \overline{\modu K}({\mathfrak P}^\mathrm{rel} \sqcup {\mathfrak F}^\mathrm{rel},\eta_i|\Lambda,{\mathbf m}) }$, for all $i \leq M$. Thus ${\modu X}$ contains $\smash{ \overline{\modu K}({\mathfrak P}^\mathrm{rel} \sqcup {\mathfrak F}^\mathrm{rel},\eta_0|\Lambda,{\mathbf m}) }$.

If ${\modu X}'$ is another irreducible component of $\smash{ \overline{\modu S}_{g}({\mathfrak W},d|\Lambda,{\mathbf m}) }$, then by the same reasoning ${\modu X}'$ also contains $\smash{ \overline{\modu K}({\mathfrak P}^\mathrm{rel} \sqcup {\mathfrak F}^\mathrm{rel},\eta_0|\Lambda,{\mathbf m}) }$ and Proposition \ref{multiplicity one} implies that ${\modu X} = {\modu X}'$. Thus $\smash{ \overline{\modu D}_{g}({\mathfrak W},d|\Lambda,{\mathbf m}) }$ is irreducible and hence so is $\smash{ \overline{\modu S}_{g}({\mathbb P}^2,d|\Lambda,\mathbf{m}) }$, completing the inductive proof. \end{proof}

\section{Appendix}

We state the two elementary results used in \S5.1 and the proof of Lemma \ref{two boats}. We continue using the convention implicit in the rest of the paper that schemes or varieties are denoted by usual Roman letters, while Deligne-Mumford stacks are denoted by calligraphic Roman letters. 

\begin{sublem}\label{general fact}
Let $\smash{ {\modu Z} \stackrel{g}{\to} Y \stackrel{f}{\to} X \stackrel{\xi}{\to} \mathrm{Spec}(\kk)}$ such that all morphisms are proper and of finite type, $X$ and $Y$ are smooth and irreducible, $f$ is smooth and all its geometric fibers are irreducible. Assume that ${\modu Z}$ is irreducible and that the geometric fibers of $g$ are generically irreducible. Then the geometric fibers of $f \circ g$ are generically irreducible. 
\end{sublem}

\begin{proof}[Sketch of proof.]
Let $x \in X(\kk)$ general. The condition that the geometric fibers of $g$ are generically irreducible implies that there exists a unique irreducible component of $(f \circ g)^{-1}(x)$ which dominates the $f$-fiber $f^{-1}(x)$. On the other hand, the condition that ${\modu Z}$ is irreducible ensures that, for general $x \in X(\kk)$, all irreducible components of $(f \circ g)^{-1}(x)$ dominate $f^{-1}(x)$, completing the proof.
\end{proof}

\begin{sublem}\label{fiber fact}
Assume that $X$ is a smooth connected complex projective variety and that the two maps $f_i:{\modu Y}_i \to X$, $i=1,2$ are proper, of finite type, dominant, have irreducible general geometric fibers and the sources ${\modu Y}_i$ are irreducible. Then there exists a unique irreducible component of ${\modu Y}_1 \times_X {\modu Y}_2$ dominating $X$, which we'll call the main component of the fibered product. 

Moreover if $y_i \in {\modu Y}_i(\cc)$, $i=1,2$, are points such that the respective fibers of $f_i$ have the expected dimension $\dim {\modu Y}_i - \dim X$ at $y_i$, then $(y_1,y_2)$ belongs to the main component of ${\modu Y}_1 \times_X {\modu Y}_2$ and no other components.
\end{sublem}

\begin{proof}[Sketch of proof.]
Let ${\modu V}_i \subseteq {\modu Y}_i$ be the open loci where the fibers have the expected dimension. It suffices to prove that ${\modu V}_1 \times_X {\modu V}_2$ is irreducible. On one hand, because the geometric fibers of ${\modu V}_i \to X$ are generally irreducible, it follows that the same holds for ${\modu V}_1 \times_X {\modu V}_2 \to X$, so there exists a unique component of ${\modu V}_1 \times_X {\modu V}_2$ dominating $X$. On the other hand, 
$$ {\modu V}_1 \times_X {\modu V}_2 = (v_1 \times v_2)^{-1}(\Delta), $$
where $\Delta \subset X^2$ is the diagonal and $v_i$ is the restriction of $f_i$ to ${\modu V}_i$, so, because the diagonal is a local complete intersection, any irreducible component ${\modu W}$ of ${\modu V}_1 \times_X {\modu V}_2$ must have at least the expected dimension $\dim {\modu Y}_1 + \dim {\modu Y}_2 - \dim X$, so it must dominate $X$ by the definition of the ${\modu V}_i$.
\end{proof}

\end{document}